\documentclass[11pt, letterpaper]{article}
\usepackage{amsfonts}
\usepackage{amssymb}
\usepackage{amsmath}
\usepackage{amsthm}
\usepackage{latexsym}
\usepackage{verbatim}
\usepackage{soul}
\usepackage{amscd}
\usepackage{mathrsfs}

\usepackage{color}

\newtheorem{thm}{Theorem}[section]
\newtheorem{lemma}[thm]{Lemma}
\newtheorem{prop}[thm]{Proposition}

\newtheorem{remark}[thm]{Remark}

\newtheorem*{thm*}{Theorem}

\theoremstyle{definition}

\numberwithin{equation}{section}

\newcommand{\lp}{\left(}
\newcommand{\rp}{\right)}

\newcommand{\R}{\mathbb{R}}
\newcommand{\RR}{\mathbb R^{n+1}_+}

\DeclareMathOperator{\divergence}{div}
\DeclareMathOperator{\re}{Re}

\newcommand{\grad}[2]{|\nabla #1|^{#2}}

\newcommand{\abs}[1]{\left\lvert #1\right\rvert}
\newcommand{\norm}[1]{\left\lVert #1\right\rVert}

\newcommand{\be}{\begin{equation}}
\newcommand{\ee}{\end{equation}}
\newcommand{\bee}{\begin{equation*}}
\newcommand{\eee}{\end{equation*}}
\newcommand{\bea}{\begin{eqnarray}}
\newcommand{\eea}{\end{eqnarray}}
\newcommand{\bs}{\begin{split}}
\newcommand{\es}{\end{split}}

\begin{document}

\title{\bf Recent progress on the fractional Laplacian in conformal geometry}
\author{Mar\'ia del Mar Gonz\'alez \thanks{M.d.M. Gonz\'alez is supported by
Spain Government project MTM2014-52402-C3-1-P and the BBVA Foundation grant for Investigadores y Creadores Culturales, and is part of the Catalan research group 2014SGR1083.}\\ Universidad Aut\'onoma de Madrid}
\date{}
\maketitle

\abstract{}

\section{Introduction}

The aim of this paper is to report on  recent development on the conformal fractional Laplacian, both from the analytic and geometric points of view, but especially towards the PDE community.

The basic tool in conformal geometry is to understand the geometry of a space by studying the transformations on such space. More precisely, let $M^n$ be a smooth Riemannian manifold of dimension $n\geq 2$ with a Riemannian metric $h$. A \emph{conformal} change of metric is such that it preserves angles so, mathematically, two metrics $h$, $\tilde h$ are conformally related if
 $$\tilde h=f h \quad\text{for some function } f>0.$$
We say that an operator $A(=A_h)$ on $M$
is \emph{conformally covariant} if under the change of metric $h_w=e^{2w} h$, then $A$ satisfies the transformation (sometimes called intertwining) law
\begin{equation}\label{intertwining-introduction}A_{h_w} \phi =e^{-bw} A_h(e^{aw}\phi)\quad\mbox{for all}\quad \phi\in \mathcal C^\infty(M),\end{equation}
for some $a,b\in\mathbb R$.
One may associate to such $A$ a notion of curvature with interesting conformal properties as
$$Q^h_A:= A_{h}(1).$$
The intertwining rule  \eqref{intertwining-introduction} then yields the $Q_A$-curvature equation
$$A_h(e^{aw})=e^{bw}  Q_A^{h_w}.$$

The most well known example of a conformally covariant operator is the conformal Laplacian
\begin{equation}\label{conformal-Laplacian}L_h :=-\Delta_h +\tfrac {n-2}{4(n-1)} R_h,\end{equation}
and its associated curvature is precisely the scalar curvature $R_h$ (modulo a multiplicative constant). The conformal transformation law is usually written as
\begin{equation}\label{conformal-property-Laplacian}L_{h_u}(\cdot)=
u^{-\frac{n+2}{n-2}} \,L_{h} (u\,\cdot\,)\end{equation}
for a change of metric
$$h_u=u^{\frac{4}{n-2}}h,\quad u>0,$$
and it gives rise to interesting semilinear equations with reaction term of power type, such as the constant scalar curvature (or Yamabe) equation
\begin{equation}\label{Yamabe-equation}
-\Delta_{h} u+ \tfrac {n-2}{4(n-1)} R_h= \tfrac {n-2}{4(n-1)} c u^{\frac{n+2}{n-2}}.
\end{equation}
The Yamabe problem (see \cite{Schoen-Yau:book,Lee-Parker} for a general background) has been one of the multiple examples of the interaction between analysis and geometry.

A higher order example of a conformally covariant operator is the Paneitz operator (\cite{Paneitz}), which is defined as the bi-Laplacian $(-\Delta_h)^2$ plus lower order curvature terms. Its associated curvature, known as $Q$-curvature, a is fourth-order geometric object that has received a lot of attention (see \cite{Chang:Zurich} and the references therein). The generalizaton to all even orders $2k$ was investigated by Graham, Jenne, Mason and Sparling ({\emph{GJMS}) in \cite{GJMS} and is based on the ambient metric construction of \cite{Fefferman-Graham}.

These operators belong to a general framework in which on the manifold $(M,h)$ there exists a meromorphic family of conformally covariant pseudodifferential operators of fractional order $$P^h_s=(-\Delta_h)^s+...\quad\quad\text{for any}\quad s\in(0,n/2).$$
$P_s^h$ will be called the \emph{conformal fractional Laplacian}. The main goal of this discussion is to describe and to give some examples, applications and open problems for this non-local object. The uniqueness issue will be postponed to Section \ref{section:uniqueness}.

 To each of these there exists an associated curvature $Q^h_s$, that generalizes the scalar curvature, the $Q$-curvature and the mean curvature. The $Q^h_s$ constitute a one-parameter family of non-local curvatures on $M$; the objective is to understand the geometric and topological information they contain, together with the study of the new non-local fractional order PDE that arise.

The conformal fractional Laplacian is defined  on the boundary of a Poincar\'e-Einstein manifold in terms of scattering theory (all the necessary background will be explained in Section \ref{section:scattering}).
Research on Poincar\'e-Einstein metrics has its origins in the work of Newman, Penrose and LeBrun \cite{LeBrun} on four dimensional space-time Physics. The subsequent work of Fefferman and Graham \cite{Fefferman-Graham} provided the mathematical framework for the study of conformally invariant (o covariant) operators on the boundary (denoted as the conformal infinity) of a Poincar\'e-Einstein manifold (the ambient) through this approach, though the study of the asymptotics of an eigenvalue problem in the spirit of Maldacena's AdS/CFT correspondence.

The celebrated AdS/CFT correspondence in string theory~\cite{Maldacena,AdS/CFT,Witten} establishes a connection between conformal field
theories in $n$ dimensions and gravity fields on an $(n +
1)$-dimensional spacetime of anti-de Sitter type, to the effect that
correlation functions in conformal field theory are given by the
asymptotic behavior at infinity of the supergravity
action. Mathematically, this involves describing the solution to the
gravitational field equations in $(n + 1)$~dimensions (which, in the
simplest case of a scalar field reduces to the Klein--Gordon equation) in
terms of a conformal field, which plays the role of the boundary data
imposed on the (timelike) conformal boundary.

An equivalent construction for $P^h_s$ has been recently proposed (\cite{Case-Chang}) in the setting of metric measure spaces in relation to Perelman's $W$-functional  \cite{Perelman}. This point of view has the advantage that it justifies the notion of a harmonic function in fractional dimensions that was sketched in the classical paper of Caffarelli and Silvestre for the usual fractional Laplacian \cite{Caffarelli-Silvestre}.\\

In Section \ref{section:extension} we will explain the construction of the conformal fractional Laplacian from a purely analytical point of view. Caffarelli and Silvestre \cite{Caffarelli-Silvestre} gave  a construction for the standard fractional Laplacian $(-\Delta_{\mathbb R^n})^s$ as a  Dirichlet-to-Neumann
operator of a uniformly degenerate elliptic boundary value problem. In the manifold case, Chang and the author \cite{Chang-Gonzalez} related the original definition of the conformal fractional Laplacian coming from scattering theory to a Dirichlet-to-Neumann operator for a related elliptic extension problem, thus allowing for an analytic treatment of
 Yamabe-type problems in the non-local setting (\cite{Gonzalez-Qing}).

 The fractional Yamabe problem, proposed in \cite{Gonzalez-Qing} poses the question of finding a constant fractional curvature metric in a given conformal class. In the simplest case, the resulting (non-local) PDE is
 \begin{equation}\label{equation-Laplacian}
(-\Delta)^s u=cu^{\frac{n+2s}{n-2s}}\quad\text{in } \mathbb R^n, \quad u>0.
 \end{equation}
 The underlying idea is to pass to the extension, looking for a solution of a (possibly degenerate) elliptic equation with a nonlinear boundary reaction term, which can be handled through a variational argument where the main difficulty is the lack of compactness. As in the usual Yamabe problem, the proof is divided into several cases; some of them still remain open.

From the geometric point of view, the fractional Yamabe problem is a generalization of Escobar's classical problem \cite{Escobar} on the construction of a constant mean curvature metric on the boundary of a given manifold, and in the particular case $s=1/2$ it reduces to it modulo some lower order error terms.\\

 We turn to examples in Sections \ref{section:sphere} and \ref{section:cylinder}. As the standard fractional Laplacian $(-\Delta_{\mathbb R^n})^s$, that can be characterized in terms of a Fourier symbol or, equivalently, as a singular integral with a convolution kernel, the conformal fractional Laplacian on the sphere $\mathbb S^n$ and on the cylinder $\mathbb R \times \mathbb S^{n-1}$ may be defined in both ways.

 Thus we first review the classical construction for the conformal fractional Laplacian on the sphere coming from representation theory, which yields its Fourier symbol, and then prove some new results on the characterization of this operator using only stereographic projection from $\mathbb R^n$. We show, in particular, a singular integral formulation for $P_s^{\mathbb S^n}$ that resembles the classical formula for the standard fractional Laplacian.

  Next we follow a parallel construction for the cylinder, recalling the results of \cite{DelaTorre-Gonzalez,DelaTorre-DelPino-Gonzalez-Wei}. More precisely, we give the explicit formula for the conformal fractional Laplacian on the cylinder in terms of its Fourier symbol, and then, a singular integral formula for a convolution kernel.

 This second example is interesting because it is the natural geometric characterization of an isolated singularity for the fractional Laplacian
 \begin{equation}\label{isolated-singularity-introduction}
\left\{\begin{split}
&(-\Delta)^{s}u=c\,u^{\frac{n+2s}{n-2s}}\text{ in }\mathbb R^n \setminus \{0\},\quad u>0,\\
&u(x)\to \infty \text{ as }|x|\to 0.
\end{split}\right.
\end{equation}
Radially symmetric solutions for \eqref{isolated-singularity-introduction} have been constructed in \cite{DelaTorre-DelPino-Gonzalez-Wei} and are known as  Delaunay solutions for the fractional curvature, since they generalize the classical construction of radially symmetric constant mean curvature surfaces \cite{Delaunay,Eells}, or radially symmetric constant scalar curvature surfaces \cite{Korevaar-Mazzeo-Pacard-Schoen,Schoen:notas}.

In addition, the cylinder is the simplest example of a non-compact manifold where the conformal fractional Laplacian may be constructed. However, being a non-local operator, it may not be well defined in the presence of general singularities. In Section \ref{section-noncompact} we give the latest development on this issue. This raises challenging questions in the area of nonlocal PDE and removability of singularities, with implications both in harmonic analysis and pseudo-differential operators.\\

Then, in Section \ref{section:uniqueness}, we consider the issue of uniqueness. Since the conformal fractional Laplacian is defined on the boundary of a Poincar\'e-Einstein manifold, it will depend on this filling. We will review here the well known
construction of two different Poincar\'e-Einstein fillings for the same boundary manifold \cite{Hawking-Page}. Unfortunately, we have not been able to find an explicit expression for the corresponding operators  $P_s^i$, $i=1,2$.\\

Our last Section \ref{section:hypersurface} is of independent interest. It is motivated by the following question: given a smooth domain $\Omega$ in $\mathbb R^{n+1}$, is there a canonical way to define the conformal fractional Laplacian on $M=\partial\Omega$ using only the information on the Euclidean metric in $\Omega$? More generally, what are the (extrinsic) conformal invariants for a hypersurface $M^n$ of $X^{n+1}$? Some invariants  have been very recently constructed in \cite{Graham:singular-Yamabe,Gover-Waldron:conformal-hypersurface-geometry,
Gover-Waldron:renormalized-volume}; these resemble the conformal non-local quantities we have defined on the boundary of a Poincar\'e-Einstein manifold but the new approach is much more general and applies to any embedded hypersurface.

In particular, when $M$ is a surface in Euclidean 3-space, one recovers the Willmore invariant with this construction (the interested reader may look at the survey \cite{CodaMarques-Neves:survey} for the latest development on the Willmore conjecture), so an interesting consequence is that one produces new (extrinsic) conformal invariants for hypersurfaces in higher dimensions that generalize the Willmore invariant of a two-dimensional surface.\\

We conclude this introduction with some remarks on further generalizations to other geometries. First, note that the formulas for the conformal fractional Laplacian in the sphere case as the boundary of the Poincar\'e ball have been long known in the representation theory community. They arise from the theory of joint eigenspaces in symmetric spaces, since the Poincar\'e model for hyperbolic space is the simplest example of a non-compact symmetric space of rank one. But there are other examples of rank one symmetric spaces: the complex hyperbolic space, that yields the CR fractional Laplacian on the Heisenberg group  or the quaternionic hyperbolic space \cite{Biquard:libro,Frank-Gonzalez-Monticelli-Tan}.

On the contrary, the picture is more complex in the higher rank case, and it is related to the theory of quantum $N$-body scattering (see \cite{Mazzeo-Vasy:symmetric-spaces} and related references). In the forthcoming paper \cite{Gonzalez-Saez} we aim to provide an analytical formulation for this problem without the represention theory machinery, when possible. The idea is to construct conformally covariant non-local operators on the boundary $M$ of a higher rank symmetric space $X^{n+k}$, which is a submanifold of codimension $k>1$. Analytically, the difficulties come from considering boundary value problems for systems of (possibly degenerate) linear partial differential equations with regular singularities (\cite{KKMOOT}).

Many of the ideas above  still hold if one switches from Riemannian to Lorentzian geometry. In particular, the conformal fractional Laplacian becomes the conformal wave operator, and one needs to move from elliptic to dispersive machinery. The papers \cite{Vasy:AdS,Enciso-Gonzalez-Vergara} provide a first approach to this setting, but many open questions still remain.


\section{Scattering theory and the conformal fractional Laplacian}
\label{section:scattering}

We first provide the general geometric setting for our construction and, in particular, the definition of a Poincar\'e-Einstein filling.

Let $X^{n+1}$ be (the interior of) a smooth Riemannian manifold of dimension $n+1$ with compact boundary
$\partial X = M^n$. A function $\rho$ is a \emph{defining function}
of $M$ in $X$ if
$$
\rho>0 \mbox{ in } X, \quad \rho=0 \mbox{ on }\partial X, \quad
d\rho\neq 0 \mbox{ on } \partial X.
$$
We say that a metric $g^+$ is \emph{conformally compact} if the new metric
$$\bar g := \rho^2g^+$$ extends smoothly to $\overline X$ for a defining function $\rho$ so
that $(\overline X,\bar g)$ is a compact Riemannian manifold. This
induces a conformal class of metrics $[h]$  on
$M$ for $h = \bar g|_{TM}$ as the defining function varies. $(M^n,[h])$  is called the
\emph{conformal infinity} and  $(X^{n+1},  g^+)$ the ambient manifold or filling.

A metric $g^+$ is
said to be \emph{asymptotically hyperbolic} if it is conformally
compact and the sectional curvature approaches to $-1$ at the conformal infinity, which is equivalent to saying that $|d\rho|_{\bar g}\to 1$. A more restrictive condition is to demand that $g^+$ is conformally compact Einstein (\emph{Poincar\'e-Einstein}), i.e., it is conformally compact and its Ricci tensor satisfies
$$Ric_{g^+}=-n g^+.$$

Given a representative $h$ of the conformal infinity $(M, [h])$,
there is a unique geodesic defining function $\rho$ such that, on
a neighborhood $M \times (0,\delta)$ in $X$, $g^+$ has the normal form
\begin{equation}
\label{normal-form}g^+ = \rho^{-2}(d\rho^2 + h_\rho)
\end{equation}
 where
$h_\rho$ is a one parameter family of metrics on $M$ such that $h_\rho|_{\rho=0}=h$ (\cite{Graham:volume-area}). In the following, we will always assume that the defining function for the problem is chosen so that the metric $g^+$ is written in normal form once the representative $h$ of the conformal infinity  is fixed. $\bar g$ will be always defined with respect to this defining function.\\

Let $(X,g^+)$ be Poincar\'e-Einstein manifold with conformal infinity $(M,[h])$. The \emph{conformal fractional Laplacian} $P_s^h$ is a nonlocal operator on $M$ which is constructed as the Dirichlet-to-Neumann operator for a generalized eigenvalue problem on $(X,g^+)$, that we describe next. Classical references are \cite{Mazzeo-Melrose,Graham-Zworski,Juhl,Guillarmou}, for instance.

The spectrum of the Laplacian $-\Delta_{g^+}$ of an asymptotically hyperbolic manifold is well known (\cite{Mazzeo-Melrose,Mazzeo:spectrum}). More precisely, it consists of the interval $\left[n^2/4,\infty\right)$ and a finite set of $L^2$-eigenvalues contained in $\left(0,n^2/4\right)$. Traditionally one writes the spectral parameter as $\sigma(n-\sigma)$; in the rest of the paper we will always assume that this value it is not an $L^2$-eigenvalue. Then, for  $\sigma\in\mathbb C$ with $\re(\sigma) >  n/2$ and such that $\sigma\not\in n/2+\mathbb N$, for each Dirichlet-type data $u\in \mathcal C^\infty(M)$, the generalized
eigenvalue problem \be\label{equation-GZ}
-\Delta_{g^+}w-\sigma(n-\sigma)w=0\quad\mbox{in } X \ee has a solution of the
form \be\label{general-solution} w = U \rho ^{n-\sigma} + \tilde U\rho^\sigma,\quad
U,\tilde U\in \mathcal C^\infty(\overline{X}),\quad U|_{\rho=0}=u. \ee
Fixed $s\in(0,n/2)$, $s\not\in \mathbb N$, and $\sigma=n/2+s$ as above, the \emph{conformal fractional Laplacian} on $M$ with respect to the metric $h$ is defined as the normalized scattering operator
\begin{equation}\label{normalized-scattering}
P^h_{s} u= d_s\,\tilde U|_{\rho=0},
\end{equation}
for the constant
\begin{equation}\label{normalization-constant}d_s=2^{2s}\frac{\Gamma(s)}{\Gamma(-s)},\end{equation}
where $\Gamma$ is the ordinary Gamma function.

Remark here that the operator $P_s^h$ is non-local, since it depends on the extension metric $g^+$ even if we do not indicate it explicitly. For the rest of this paper we will always assume that a background metric $g^+$ has been fixed.

The main properties of the conformal fractional Laplacian are summarized in the following:
\begin{itemize}
\item[\emph{i.}] $P^h_s$ is a self-adjoint pseudo-differential operator on $M$ with principal symbol the same as $(-\Delta_{h})^{s}$, i.e.,
   $$P_s^h\in (-\Delta_h)^s+\Psi_{s-1},$$
   where $\Psi_{l}$ is the set of pseudo-differential operators of loss $l$.

   \item[\emph{ii.}] In the case that $M=\mathbb R^n$ with the Euclidean metric $|dx|^2$ and its canonical extension to $\mathbb R^{n+1}_+$, all the curvature terms vanish and
       $$P_s^{\mathbb R^n}=(-\Delta_{\mathbb R^n})^s,$$
       i.e., we recover the classical fractional Laplacian.

\item[\emph{iii.}] $P_s^h$ is a conformally covariant operator, in the sense that under the conformal change of metric
$$h_u=u^{\frac{4}{n-2s}}h,\quad u>0,$$
it satisfies the transformation law
\begin{equation}\label{conformal-covariance}P_s^{h_u}(\cdot)=
u^{-\frac{n+2s}{n-2s}} \,P_s^{h} (u\,\cdot\,).\end{equation}
\end{itemize}

The \emph{fractional order curvature} of the metric $h$ on $M$ associated to the conformal
fractional Laplacian $P_s^{h}$ is defined as
$$Q_s^{h} = P_s^{h}(1),$$
although note that other authors use a different normalization constant.
From the above relation
\eqref{conformal-covariance} we obtain the  curvature equation
\be\label{curvature-equation}
P_{s}^{h}(u)= Q_s^{h_u}\,u^{\frac{n+2s}{n-2s}}\quad \text{in }M^n, \ee
which is a non-local semilinear equation with critical power nonlinearity generalizing \eqref{equation-Laplacian} to the curved case.

One of the main observations is that the $P_s^h$ constitute a one-parameter meromorphic family of conformally covariant operators on $M$, for $s\in(0,n/2)$, $s\not\in \mathbb N$. At the integer powers, the conformal $s$-Laplacian can be constructed  by a residue formula thanks to the normalization constant \eqref{normalization-constant} (see \cite{Graham-Zworski}). In addition, when $s$ is a positive integer, $P^h_s$ is a local operator that coincides with the classical GJMS operator from \cite{GJMS,Fefferman-Graham}. In particular:
\begin{itemize}
\item For $s = 1$, $P_1$ is precisely the conformal Laplacian defined in \eqref{conformal-Laplacian}, i.e.,
\begin{equation}\label{conformal-Laplacian1}
P_1^{h} = L_h=-\Delta_{h} + \tfrac {n-2}{4(n-1)}R_h,
\end{equation}
and the associated curvature is a multiple of the scalar curvature
$$Q_1^{h} = \tfrac {n-2}{4(n-1)} R_h.$$

\item For $s=2$, the conformal fractional Laplacian coincides with the well known Paneitz operator (\cite{Paneitz})
$$P_2^{h}=(-\Delta_{h})^2+\delta(a_nR_{h}+b_nRic_{h})d+\tfrac{n-4}{2}\,Q^h_{2},$$
and $Q_2$ is (up to multiplicative constant) the so-called fourth order $Q$-curvature.
\end{itemize}
For any other powers $s\not\in \mathbb N$, $P^h_s$ is a non-local operator on $M$ and reflects the geometry of the filling $(X,g^+)$. Some explicit examples will be considered in Sections \ref{section:sphere} and \ref{section:cylinder}.\\

We also remark that the same construction is true for a general asymptotically hyperbolic manifold, except for values $s\in\mathbb N/2$, unless the expansion of the term $h_\rho$ in the normal form \eqref{normal-form} is even up to a suitable order \cite{Graham-Zworski,Guillarmou}. In this exposition we will explain in detail only the case $s\in(0,1)$, where we will explain the role played by mean curvature (see Theorem \ref{thm-extension2} below).

\section{The extension and the $s$-Yamabe problem}\label{section:extension}

It was observed in \cite{Chang-Gonzalez} (see also
\cite{Case-Chang} for the most recent development)
that the generalized eigenvalue problem
\eqref{equation-GZ}-\eqref{general-solution} on $(X, g^+)$ is
equivalent to a linear degenerate elliptic problem on the compactified
manifold $(\overline X, \bar g)$. Hence
they reconciled the definition of the conformal fractional
Laplacian $P_s^h$ given  in the previous section as the normalized scattering operator and
the one given in the spirit of the Dirichlet-to-Neumann operators by
Caffarelli and Silvestre in \cite{Caffarelli-Silvestre}.

In this section we will assume that $s\in(0,1)$.  For higher powers $s>1$ we refer to \cite{Case-Chang}, \cite{Yang} and \cite{Chang-Yang}. As in the introduction chapter, we set $a=1-2s$. $W^{1,2}(X,\rho^a)$ will denote the weighted Sobolev space $W^{1,2}$ on $X$ with weight $\rho^a$.

\begin{thm}
[\cite{Chang-Gonzalez}]\label{thm-Chang-Gonzalez}
Let  $(X, g^+)$  be a Poincar\'e-Einstein manifold with conformal infinity $(M,[h])$.
Then, given $u\in \mathcal C^{\infty}(M)$, the generalized eigenvalue
problem \eqref{equation-GZ}-\eqref{general-solution} is equivalent
to the degenerate elliptic equation
\begin{equation}\label{div}\left\{\begin{split}
-\divergence \lp \rho^a \nabla U\rp + E(\rho) U &=0\quad \mbox{in }(X,\bar g), \\
U|_{\rho=0}&=u\quad \mbox{on }M,
\end{split}\right.\ee
where the derivatives are taken with respect to the original metric $\bar g$, and $U = \rho^{n/2-s} w$. The zero-th order term is
\begin{equation*}\label{E1} E(\rho)= \rho^{-1-\sigma} \lp -\Delta_{g^+} -
\sigma(n-\sigma)\rp\rho^{n-\sigma}.\end{equation*}
Notice that, in a neighborhood
$M\times(0,\delta)$ where the metric $g^+$ is in normal form \eqref{normal-form}, this expression simplifies to
\be\label{E-normal-form} E(\rho)=\tfrac{n-1+a}{4n}  R_{\bar g}
\rho^a. \ee
Such $U$ is the unique minimizer of the
energy
$$F[V] = \int_X \rho^a |\nabla V|_{\bar g}^2\,dv_{\bar g} + \int_X E(\rho)|V|^2\,dv_{\bar g}$$
among all the functions $V\in W^{1,2}(X,\rho^a)$ with fixed trace
$V|_{\rho=0}=u$.
Moreover, we recover the conformal fractional Laplacian on $M$ as
\begin{equation*}
P_s^{h} u=
-d^*_s\lim_{\rho\to 0}\rho^a \partial_\rho U,
\end{equation*}
where
\begin{equation*}d^*_s =-
\frac{2^{2s-1}\Gamma(s)}{s\Gamma(-s)}.
\end{equation*}
\end{thm}

Before we continue with the exposition, let us illustrate these concepts with the simplest example of a Poincar\'e-Einstein manifold: the hyperbolic space $\mathbb H^{n+1}$. It can be
characterized as the upper half-space $\mathbb R^{n+1}_+$ (with coordinates $x\in\mathbb
R^n$, $y\in \mathbb R_+$), endowed with the  metric
$$g^+=\frac{dy^2+|dx|^2}{y^2}.$$
In this case, $y$ is a defining function and the conformal infinity $\{y=0\}$ is just the Euclidean space $\mathbb R^n$ with its flat metric $|dx|^2$.
Then problem \eqref{div} with Dirichlet condition $u$ reduces to
\begin{equation}\label{div-Euclidean}\left\{\begin{split}
-\divergence \lp y^a \nabla U\rp &=0 \quad\mbox{in }\mathbb R^{n+1}_+,\\
U|_{y=0}&=u \quad\mbox{on }\mathbb R^n,
\end{split}\right.\end{equation}
and the fractional Laplacian at the boundary $\mathbb R^n$ is just
$$
P_s^{\mathbb R^n} u = (-\Delta_{\mathbb R^n})^s u = -
d_s^*\,\lim_{y\to 0} \lp y^a \partial_y U\rp,
$$
which is precisely the usual construction for the fractional Laplacian as a Dirichlet-to-Neumann operator from \cite{Caffarelli-Silvestre}. We note that it is possible to write $U=P *_x u$, where $P$ is the Poisson kernel for this extension problem as given in the introduction chapter.\\

If the background manifold $(X,g^+)$ is not Poincar\'e-Einstein, but only asymptotically hyperbolic, we have a similar extension problem but here the mean curvature of the boundary $M$
respect to the metric $\bar g$ in $\overline X$, denoted by $H$,  plays an essential role:

\begin{thm}[\cite{Chang-Gonzalez}]\label{thm-extension2}
Let $(X^{n+1},\bar g^+)$ be an asymptotically hyperbolic manifold with a geodesic defining function $\rho$ and conformal infinity $(M^n,[h])$, with the metric written in normal form \eqref{normal-form}. Then the conformal fractional Laplacian can be constructed through the following extension problem:  for each given smooth function $u$ on $M$, consider
\begin{equation*}\left\{\begin{split}
-\divergence(\rho^a\nabla U)+E(\rho)  U&=0\quad\mbox{in }(\bar X,\bar g), \\
U&=u \quad \mbox{on }M,
\end{split}\right.\end{equation*}
where
\begin{equation}\label{E-general} E(\rho)=\tfrac{n-1-a}{4n} \left[ R_{\bar g}
- \left\{n(n+1) + R_{g^+}\right\}\rho^{-2} \right] \rho^a.\end{equation}
Then there exists a unique solution $U$ and moreover,
\begin{enumerate}
\item For $s\in(0,\frac{1}{2})$,
\be\label{compute-fractional}P_s^h u=-{d_s^*}\lim_{\rho\to 0}\rho^a \partial_\rho U,\ee

\item For $s=\frac{1}{2}$, we have an extra term
$$P_{\frac{1}{2}}^h u=\lim_{\rho\to 0} \partial_\rho U + \tfrac{n-1}{2} H u.$$

\item If $s\in\lp\frac{1}{2},1\rp$, the limit in the right hand side of \eqref{compute-fractional} exists if and only if the the mean curvature $H$ vanishes identically, in which case, \eqref{compute-fractional} holds too.
\end{enumerate}
\end{thm}

\begin{remark}
Note that, in the particular case that $s=1/2$, the fractional curvature $Q_{1/2}$ reduces to the mean curvature of $M$ (up to multiplicative constant).
\end{remark}

This dichotomy in Theorem \ref{thm-extension2} due to the presence of mean curvature also appears in many other non-local problems such as \cite{Caffarelli-Souganidis:convergence-nonlocal,Gonzalez:Gamma-convergence,Savin-Valdinoci}. The underlying idea is that, for $s\in(0,1/2)$, non-local curvature is the essential term, but for $s\in(1/2,1)$, mean curvature takes over.\\

Our next objective is, in the Poincar\'e-Einstein case, to compare the geometric extension \eqref{div} to the Euclidean one \eqref{div-Euclidean}. It was observed in \cite{Chang-Gonzalez} that it is possible to find a special defining function $\rho^*$ such that when we rewrite the scattering equation \eqref{div} for the new metric $\bar g^*=(\rho^*)^2 g^+$, the lower order term $E(\rho^*)$ vanishes; thus making the extension as close as possible to the Euclidean one. This construction was inspired in the special defining function of \cite{Lee:spectrum} and is also an essential ingredient in the formulation of the scattering problem in the metric measure space setting of \cite{Case-Chang}, which is a very interesting development for which we do not have space here.

As it was pointed out in \cite{Case-Chang}, it is necessary to assume, here and in the rest of this exposition, that the first eigenvalue for $-\Delta_{g^+}$ satisfies  $\lambda_1(-\Delta_{g^+})>\frac{n^2}{4}-s^2$.  We arrive at the following
improvement of Theorem \ref{thm-Chang-Gonzalez}:

\begin{prop}[\cite{Chang-Gonzalez,Case-Chang}]\label{new-defining-function}
Let $w^*$ be the solution to \eqref{equation-GZ}-\eqref{general-solution} with Dirichlet data $u\equiv 1$, and set $\rho^* = (w^*)^\frac 1{n/2-s}$. The function $\rho^*$ is a defining function of $M$ in $X$ such that
$E(\rho^*)\equiv 0$. Moreover, it has the asymptotic expansion near the conformal infinity
$$\rho^*(\rho)=\rho\left[ 1+\frac{Q^{h}_s}{(n/2-s)d_s}  \rho^{2s}+O(\rho^2)\right].$$
By construction, if $U^*$ is the solution to
\begin{equation*}\label{new-extension}\left\{
\begin{split}
-\divergence \lp (\rho^*)^a \nabla U^*\rp&=0 \quad \mbox{in }(X, \bar g^*),\\
U^*&=u \quad\mbox{on }M,
\end{split}\right.\end{equation*}
with respect to the metric $\bar g^*=(\rho^*)^2 g^+$, then
 \begin{equation*}\label{property10} P_s^{h} u=
-d^*_s \lim_{{\rho^*}\to 0} (\rho^*)^a\,\partial_{\rho^*} U^*+u\,
Q_s^{h}.\end{equation*}
\end{prop}

Next, we give an interpretation of the fractional curvature as a variation of weighted volume, in analogy to the usual mean curvature situation. The notion of renormalized volume was first investigated by the physicists in relation to the AdS/CFT correspondence, and was considered by  \cite{Fefferman-Graham:Q-curvature,Chang-Qing-Yang}. Given $(X^{n+1},g^+)$  a Poincar\'e-Einstein manifold with boundary $M^n$ and defining function $\rho$, one may compute the asymptotic expansion of the volume of the region $\{\rho>\epsilon\}$; the renormalized volume is defined as one very specific term in this asymptotic expansion.
When the dimension $n$ is odd, the renormalized volume is a conformal
invariant of the conformally compact structure, and it can be calculated as the conformal primitive of the $Q$-curvature coming from the scattering operator (this is the case $s=n/2$). In that case that $n$ is even, the picture is more complex, and one can show that the renormalized volume is one term of the Chern-Gauss-Bonnet formula in higher dimensions.

When $s\in(0,1)$ one can also give a weighted version for volume (see \cite{Gonzalez}), and to obtain the fractional curvature $Q_s$ as its first variation. More precisely, for each $\varepsilon>0$ we set
 \be\label{weighted-volume}vol_{g^+,s}(\{\rho>\varepsilon\}):=
\int_{\{\rho>\varepsilon\}}(\rho^*)^{\frac{n}{2}-s}\,dv_{g^+},\ee
where $\rho^*$ is the special defining function from  Proposition \ref{new-defining-function}.

\begin{prop}[\cite{Gonzalez}]
Let $(X, g^+)$ be a Poincar\'e-Einstein manifold with conformal infinity $(M,[h])$. The weighted volume \eqref{weighted-volume} has an asymptotic expansion in $\varepsilon$ when $\varepsilon \to 0$ given by
$$vol_{g^+,s}(\{\rho>\varepsilon\})=\varepsilon^{-\frac{n}{2}-s}\left[\lp\tfrac{n}{2}+s\rp^{-1}vol(M)+ \varepsilon^{2s} V_s^h +\mbox{ higher order terms}\right]$$
where
$$V_s^h:=\frac{1}{d_s}\frac{1}{n/2-s}\int_{M} Q_s^{h}\,dv_{h}.$$\\
\end{prop}


Finally, and as an application of the extension Theorems \ref{thm-Chang-Gonzalez} and \ref{thm-extension2}, we give a summary of the recent development on the fractional Yamabe problem.

The resolution of the classical Yamabe problem by Aubin, Schoen, Trudinger, has been one of the most significant advances in geometric analysis (see \cite{Lee-Parker,Schoen-Yau:book}, and the references therein). Given a smooth background  metric, the problem is to find a conformal one that has constant scalar curvature. In PDE language, this is \eqref{Yamabe-equation}.

One may  pose then the analogous question of finding a constant $Q_s$-curvature in the same conformal class as a given one. This study was initiated by the author and Qing in \cite{Gonzalez-Qing}, and it amounts to, given a background metric $(M^n,h)$, solve the following non-local semilinear geometric equation with critical exponent (recall \eqref{curvature-equation}),
\begin{equation}
\label{yamabe} P_s^{h}(u)=c
u^{\frac{n+2s}{n-2s}}, \quad u>0, \end{equation}
for some constant
$c$ on $M$.

Theorem \ref{thm-Chang-Gonzalez} allows to write \eqref{yamabe} as a local elliptic equation in the extension with a non-linear boundary reaction term:
\begin{equation}\label{Yamabe1}\left\{\begin{split}
-\divergence(\rho^a U)+E(\rho) U&=0 \quad\text{in }(X^{n+1},\bar g),\\
-d^*_s\lim_{\rho\to 0}\rho^a \partial_\rho U&=cu^{\frac{n+1}{n-1}}\quad \text{on } M^n,\quad u>0,
\end{split}\right.\end{equation}
where we have written $u=U(\cdot,0)$.

Even though \eqref{yamabe} is a non-local equation, the resolution to the fractional Yamabe problem follows the same scheme as in the original Yamabe problem for the scalar curvature, using a variational method. In addition, the $s=1/2$ case is deeply related to the prescribing constant mean curvature problem (also known as the boundary Yamabe problem) considered by Escobar \cite{Escobar}, Brendle-Chen \cite{Brendle-Chen}, Li-Zhu \cite{Li-Zhu}, Marques \cite{Marques1,Marques2}, Almaraz \cite{Almaraz}, Mayer-Ndiaye \cite{Mayer-Ndiaye1} and others, and which corresponds to the following Dirichlet-to-Neumann operator
\begin{equation}\label{boundary-Yamabe}\left\{\begin{split}
-\Delta_{\bar g} u+\tfrac{n-1}{4n} R_{\bar g} u&=0 \quad\text{in }(X^{n+1},\bar g),\\
\partial_\nu u+\tfrac{n-1}{2}Hu&=cu^{\frac{n+1}{n-1}}\quad \text{on } M^n.
\end{split}\right.\end{equation}
This connection will be made precise right below \eqref{connection}. However, there is a subtle issue: in the proof one will need to find  particular background metric $(X,\bar g)$ with very precise asymptotic behavior near a point $p\in M$ in a good coordinate system. However, in contrast to the study of \eqref{boundary-Yamabe}, where they are free to choose conformal Fermi coordinates on the filling $(\overline X,\bar g)$, our freedom of choice of metrics for \eqref{Yamabe1} is restricted to the boundary. Once a metric  $h_1\in[h]$ is chosen, the corresponding defining function $\rho_1$ is determined and the extension metric $\bar g_1$, written in normal form \eqref{normal-form} for $\bar g_1=(\rho_1)^2 g^+$, is unique and cannot be simplified.

Let us set up the notation for \eqref{yamabe}. We consider a scale-free functional on metrics in the class $[h]$ on $M$
given by
\begin{equation*}
\label{functional0} I_s[h] = \frac {\int_M
Q^{h}_s \,dv_{h}}{(\int_M \,dv_{h})^\frac
{n-2s}n}. \end{equation*}
 Or, if we set a base metric $h$ and write a
conformal metric $h_u=u^{\frac{4}{n-2s}}h$,
then
\begin{equation*}\label{functional1} I_s[u, h]=\frac{\int_M u
P_s^{ h} (u) \,dv_{h}}{ \lp \int_M u^{2^*}\,dv_{h}\rp^{\frac{2}{2^*}}},
\end{equation*} where
$$2^*=\frac{2n}{n-2s}.$$
We will
call $I_s$ the $s$-Yamabe functional.

Our objective is to find a metric in the
conformal class  $[h]$ that minimizes the
functional $I_s$. It is clear that a metric $h_u$, where
$u$ is a minimizer of $I_s[u,h]$,  is an admissible solution for \eqref{yamabe} (positivity will be guaranteed by an application of a suitable maximum principle).
 This suggests that we define the $s$\emph{-Yamabe constant}
\begin{equation*}
\label{y-constant}\Lambda_s (M,[h])=\inf \left\{
I_s[h_1] : h_1\in [h]\right\}.\end{equation*}
It is then apparent that $\Lambda_s(M,[ h])$ is an invariant on the conformal class $[h]$ when $g^+$ is fixed. In addition, it is proved in \cite{Gonzalez-Qing} that the sign of $\Lambda_s(M,[h])$ governs the sign of the possible constants $c$ in \eqref{yamabe}, and the sign of the first eigenvalue for $P_s^h$.

In the mean time, based on Theorem \ref{thm-Chang-Gonzalez}, we
set \begin{equation}\label{nicer-functional} \bar I_s[U,\bar
g]=\frac{\int_X \rho^a\abs{\nabla U}_{\bar g}^2
\,dv_{\bar g} + \int_X E(\rho)\,U^2\,dv_{\bar g}}{\lp \int_M
|U|^{2^*} \,dv_{h}\rp^{\frac{2}{2^*}}}.\end{equation}
Note then  \be\label{good-gamma-Yamabe}
\Lambda_s (X,  [h])=\inf \left\{ \bar I_s[U,\bar g] :
U\in W^{1,2}(X,\rho^a) \right\}. \ee
As a consequence,  fixing the integral $\int_M u^{2^*} \,dv_{h} = 1$, if $U$ is a minimizer of the functional $\bar I_s[\,\cdot\,,\bar g]$, then its trace $u=U(\cdot,0)$ is a solution for \eqref{yamabe}.

This minimization procedure is related to the trace Sobolev embedding
$$W^{1,2}(X,\rho^a)\hookrightarrow H^{s}(M)\hookrightarrow L^{2^*}(M),$$
which is continuous, but not compact. Hence the difficulty comes from this lack of compactness, which is well understood in the Euclidean case below:

\begin{thm}[\cite{Lieb:sharp-constants}]\label{Sobolev-embedding}
There exists a positive constant $C_{n,s}$ such that for every function  $U$ in $W^{1,2}(\RR,y^a)$  we have that
\begin{equation*} \norm{u}_{L^{2^*}(\R^n)}^2\leq  C_{n,s}
\int_{\RR} y^a \grad U 2\,dx\,dy, \end{equation*}
where $u$ is the trace $u:=U(\cdot,0)$. Moreover equality holds if and only if $u$ is a ``bubble", i.e.,
\begin{equation}\label{bubbles}u(x) = C \lp \frac{\mu}
{\abs{x-x_0}^2+\mu^2}\rp^{\frac{n-2s}{2}},\quad x\in \mathbb R^n,
\end{equation}
for $C\in\mathbb R$, $\mu>0$ and $x_0\in\R^n$ fixed, and $U=P*_x u$ its Poisson extension.
\end{thm}
 We also remark that all entire positive solutions to
 $$(-\Delta)^s u=u^{\frac{n+2s}{n-2s}},\quad u>0,$$
 have been completely classified (see \cite{JinLiXiong}, for instance, for an account of references). In particular, they must be the standard ``bubbles" \eqref{bubbles}. Other non-linearities for fractional Laplacian equations have been considered, for instance in  \cite{Cabre-Tan:positive-solutions,Cabre-Sire:I,Servadei-Valdinoci:Brezis-Nirenberg,
 Barrios-Colorado-DePablo-Sanchez}, although by no means this list is exhaustive. \\

Going back to the minimization problem \eqref{good-gamma-Yamabe}, we observe that the variational method that was used in the resolution of the classical Yamabe problem can still be applied, but the difficulty comes from the specific structure of metric in the filling $(X,\bar g)$. In any case, one starts by comparing the Yamabe constant on $M$ to the Yamabe constant on the sphere.

Using stereographic projection, from Theorem \ref{Sobolev-embedding} it is easily seen that
$$
\Lambda_s(\mathbb S^n, [h_{\mathbb S^n}]) = \frac{1}{C_{n,s}},
$$
where $[h_{\mathbb S^n}]$ is the canonical conformal class of metrics on the
sphere $\mathbb S^n$ understood as the conformal infinity of the Poincar\'e ball.

Suppose that $(X^{n+1}, g^+)$ is an
asymptotically hyperbolic manifold with a geodesic defining function $\rho$ and set $\bar g=\rho^2 g^+$. Let $(M^n, [h])$ be its conformal infinity. One can show that (\cite{Gonzalez-Qing,Case:energy-inequalities}) the fractional Yamabe constant satisfies
$$-\infty<\Lambda_s(M,[h])\leq \Lambda_s(\mathbb S^n,[h_{\mathbb S^n}]).$$

\begin{thm}[\cite{Gonzalez-Qing}]
In the setting above, if
\begin{equation}\label{condition}
\Lambda_s(M,[h])<\Lambda_s(\mathbb S^n,[h_{\mathbb S^n}]),
\end{equation} then the
$s$-Yamabe problem is solvable for $s \in (0, 1)$.
\end{thm}

Therefore, it suffices to find a suitable test function in the functional \eqref{nicer-functional} that attains this strict inequality. As we have mentioned, one needs to find suitable conformal normal coordinates on $M$ by conformal change, and then deal with the corresponding extension metric. Hence one needs to make some assumptions on the behavior of the asymptotically hyperbolic manifold $g^+$ . The underlying idea here is to have $g^+$ as close as possible as a Poincar\'e-Einstein manifold. The first one of these assumptions is
$$R_{g^+}+n(n+1)=o(\rho^2) \quad\text{as}\quad \rho\to 0,$$
which looks very reasonable in the light of \eqref{E-general}. In particular, under this condition one has that
\begin{equation}\label{connection}E(\rho)=\tfrac{n-1+a}{4n}R_{\bar g}\rho^a+o(\rho^2)\quad\mbox{as}\quad\rho\to 0.\end{equation}
(compare to \eqref{E-normal-form}). Another consequence of this expression is that the $1/2$-Yamabe problem coincides to the prescribing constant mean curvature problem \eqref{boundary-Yamabe}, up to a small error.
In general one needs a higher order of vanishing for $g^+$ (see \cite{Kim-Musso-Wei} for the precise statements), which is automatically  true if $g^+$ is Poincar\'e-Einstein and not just asymptotically hyperbolic. This also shows that the natural geometric setting for an asymptotically hyperbolic $g^+$ is to demand that $g^+$ has constant scalar curvature $R_{g^+}=-n(n+1)$.

The first attempt to prove \eqref{condition} was \cite{Gonzalez-Qing} in the non-umbilic case, where the authors use a bubble as a test function. The umbilic, non-locally conformally flat case in high dimensions was considered in  \cite{Gonzalez-Wang}. Finally, Kim, Musso and Wei \cite{Kim-Musso-Wei} have provided an unified development, covering all the cases that do not need a positive mass theorem for the conformal fractional Laplacian. Their test function is not a ``bubble" but instead it has a  more complicated geometry. Summarizing, some hypothesis under which the fractional Yamabe problem for $s\in(0,1)$ is solvable (in addition to those on $g^+$ above) are:
\begin{itemize}
\item $n\geq 2$, $s\in(0,1/2)$, $M$ has a point of negative mean curvature.
\item $n\geq 4$, $s\in(0,1)$, $M$ is not umbilic.
\item $n>4+2s$, $M$ is umbilic but not locally conformally flat.
\item $M$ is locally conformally flat or $n=2$, and the fractional positive mass theorem holds.
\end{itemize}
However, we see from this last point that to cover all the cases with this method one still needs to develop a positive mass theorem for the Green's function of the conformal fractional Laplacian, which is at this time a puzzling open question. From another point of view, we mention the work \cite{Mayer-Ndiaye}, where they use the the barycenter technique of Bahri-Coron to bypass the positive mass issue for the locally flat and umbilic conformal infinity.

Finally, one may look at the lack of compactness phenomenon. In general, Palais-Smale sequences can be decomposed into the solution of the limit equation plus a finite number of bubbles. Moreover, the multi-bubbles are non-interfering even though the operator is non-local (see, for instance, \cite{Fang-Gonzalez,Palatucci-Pisante,Kim-Musso-Wei:noncompactness,Kim-Musso-Wei:compactnessI}).



\section{The conformal fractional Laplacian on the sphere}\label{section:sphere}

In this section we look at the sphere $\mathbb S^n$ with the round metric $h_{\mathbb S^n}$, understood as the conformal infinity of the Poincar\'e ball model for hyperbolic space $\mathbb H^{n+1}$. Note that hyperbolic space is the simplest example of a Poincar\'e-Einstein manifold, and the model for the general development.

On $\mathbb S^n$ one explicitly knows (\cite{Branson:sharp-inequalities}, see also the lecture notes \cite{Branson}, for instance) that the conformal fractional Laplacian (or intertwining operator) has the explicit expression
\begin{equation}\label{intertwining-sphere}P^{\mathbb S^n}_s=\frac{\Gamma\lp A_{1/2}+s+\tfrac{1}{2}\rp}{\Gamma\lp A_{1/2}-s+\tfrac{1}{2}\rp}, \quad
A_{1/2}=\sqrt{-\Delta_{\mathbb S^n}+\lp\tfrac{n-1}{2}\rp^2},\end{equation}
for all $s\in(0,n/2)$.  From here one easily calculates  that the fractional curvature of the sphere is a positive constant
\begin{equation}\label{curvature-sphere}
Q^{\mathbb S^n}_s=P_s^{\mathbb S^n}(1)=\frac{\Gamma\lp \frac{n}{2}+s\rp}{\Gamma\lp \frac{n}{2}-s\rp}.
\end{equation}
Formula \eqref{intertwining-sphere} may be easily derived from the scattering problem \eqref{equation-GZ}-\eqref{general-solution}. A proof can be found in the book \cite{Baum-Juhl:book}, which also makes the link to the representation theory community. Note, however, a different factor of 2, which is always an issue when passing from representation theory to geometry. For convenience of the reader not familiar with this subject we provide a direct proof below.

Consider the Poincar\'e metric for
hyperbolic space $\mathbb H^{n+1}$, written in normal form \eqref{normal-form} as
$$g^+=\rho^{-2}\lp d\rho^2+\big(1-\tfrac{\rho^2}{4}\big)^2h_{\mathbb S^n}\rp,$$
for $\rho\in(0,2]$. Remark that $\rho=2$ corresponds to the origin of the Poincar\'e ball and thus the apparent singularity is just a consequence of the expression for the metric in polar-like coordinates.

Calculating the Laplace-Beltrami operator with respect to $g^+$ we obtain, recalling that $\sigma=\frac{n}{2}+s$, that the eigenvalue equation \eqref{equation-GZ} is equivalent to the following:
\begin{equation}\label{ode1}
\rho^{n+1}\lp 1-\tfrac{\rho^2}{4}\rp^{-n}\partial_\rho\left[ \rho^{-n+1}\left(1-\tfrac{\rho^2}{4}\right)^n\partial_\rho w\right]+\rho^2\left(1-\tfrac{\rho^2}{4} \right)^{-2}\Delta_{\mathbb S^n}w+\lp\tfrac{n^2}{4}-s^2\rp w=0.
\end{equation}
We will show that the operator $P_s^{\mathbb S^n}$ diagonalizes
in the spherical harmonic decomposition for $\mathbb S^{n}$. With some abuse of notation, let $\mu_m=m(m+n-1)$, $m=0,1,2,...$ be the eigenvalues of $-\Delta_{\mathbb S^{n}}$, repeated according to multiplicity, and $\{E_m\}$ be the corresponding basis of eigenfunctions. The projection of \eqref{ode1} onto each eigenspace $\langle E_m\rangle$ yields
\begin{equation*}\label{ode2}
\rho^{n+1}\lp 1-\tfrac{\rho^2}{4}\rp^{-n}\partial_\rho\left[ \rho^{-n+1}\left(1-\tfrac{\rho^2}{4}\right)^n\partial_\rho w_m\right]-\rho^2\left(1-\tfrac{\rho^2}{4} \right)^{-2} \mu_m w_m+\lp\tfrac{n^2}{4}-s^2\rp w_m=0.
\end{equation*}
This is a hypergeometric ODE with general solution
\begin{equation}\label{w_m}
w_m(\rho)=c_1 \rho^{\frac{n}{2}-s}\varphi_1(\rho)+c_2 \rho^{\frac{n}{2}+s}\varphi_2(\rho), \quad c_1,c_2\in\mathbb R,
\end{equation}
for
\begin{equation*}
\begin{split}
&\varphi_1(\rho):= (\rho^2-4)^{\frac{-n-\beta+1}{2}}\,_2F_1 \lp \tfrac{-\beta+1}{2},
\tfrac{-\beta+1}{2}-s,1-s,\tfrac{\rho^2}{4}\rp,\\
&\varphi_2(\rho):= (\rho^2-4)^{\frac{-n-\beta+1}{2}}\,_2F_1 \lp \tfrac{-\beta+1}{2},
\tfrac{-\beta+1}{2}+s,1+s,\tfrac{\rho^2}{4}\rp,
\end{split}\end{equation*}
where we have defined
$$\beta:=\sqrt{(n-1)^2+4\mu_m}$$
and $\,_2F_1$ is the usual Hypergeometric function.

In order to calculate the conformal fractional Laplacian, first one needs to obtain an asymptotic expansion of the form \eqref{general-solution} for $U,\tilde U$ smooth up to $\overline X$. Since $w$ must be smooth at the central point $\rho=2$, one should choose the constants $c_1,c_2$ such that in \eqref{w_m} the singularities of $\varphi_1$ and $\varphi_2$ at $\rho=2$ cancel out. This is,
\begin{equation}\label{equation100}c_1 2^{\frac{n}{2}-s} \,_2F_1 \lp \tfrac{-\beta+1}{2},
\tfrac{-\beta+1}{2}-s,1-s,1\rp
+c_2 2^{\frac{n}{2}+s}\,_2F_1 \lp \tfrac{-\beta+1}{2},
\tfrac{-\beta+1}{2}+s,1+s,1\rp=0.\end{equation}
In order to simplify this expression, recall the following property of the Hypergeometric function from \cite{Abramowitz-Stegun}: if $a+b<c$, then
$$_2F_1(a,b,c,1)=\frac{\Gamma(c)\Gamma(c-a-b)}{\Gamma(c-a)\Gamma(c-b)}.$$
After some calculation, \eqref{equation100} yields
\begin{equation}\label{equation101}
\frac{c_2}{c_1}= 2^{-2s}\frac{\Gamma\big(\tfrac{1}{2}+s+\frac{\beta}{2}\big) \Gamma(-s) }{\Gamma\big(\frac{1}{2}-s+\frac{\beta}{2}\big) \Gamma(s)}.
\end{equation}

 Next, looking at the definition of the conformal fractional Laplacian from \eqref{normalized-scattering}, and noting that both $\varphi_1,\varphi_2$ are smooth at $\rho=0$, we conclude from \eqref{equation101} that
$$P_s^{\mathbb S^n}|_{\langle E_m\rangle} u_m= d_s\frac{c_2}{c_1}u_m= \frac{\Gamma\big(\frac{1}{2}+s+\frac{\beta}{2}\big)}
{\Gamma\big(\frac{1}{2}-s+\frac{\beta}{2}\big)} u_m. $$
This concludes the proof of \eqref{intertwining-sphere} when $s\in(0,n/2)$ is not an integer.\\

For integer powers $k\in\mathbb N$, it can be shown that  \eqref{intertwining-sphere} also yields the factorization formula for the GJMS operators on the sphere
\begin{equation}\label{factorization-integers}P_k^{\mathbb S^n}=\prod_{j=1}^k \left\{ -\Delta_{\mathbb S^n}+\lp\tfrac{n}{2}+j-1\rp\lp\tfrac{n}{2}-j\rp\right\}.\end{equation}
The paper \cite{Graham:conformal-powers} by Graham independently derives this expression just by using the formula for the corresponding operator on Euclidean space $\mathbb R^n$ and then stereographic projection to translate it back to the sphere $\mathbb S^n$.

Here we show that Graham's method using stereographic projection also works for non-integer $s$, yielding a factorization formula in the spirit of \eqref{factorization-integers}. The advantage of this formulation is that it does not require the extension, but only the conformal property  \eqref{conformal-covariance} from Euclidean space to the sphere.

\begin{prop}\label{thm-formula-induction-sphere} Let $ s _0\in(0,1)$, $k\in\mathbb N$. Then
$$P^{\mathbb S^n}_{k+ s _0}=\prod_{j=1}^k \lp P_1^{\mathbb S^n}+c_j\rp P_{ s _0}^{\mathbb S^n}.$$
for $c_j=-(s _0+j-1)( s _0+j)$.
\end{prop}
Here $P_1^{\mathbb S^n}$ is the usual conformal Laplacian \eqref{conformal-Laplacian1} on $\mathbb S^n$, i.e.,
\be\label{conformal-Laplacian-sphere}P_1^{\mathbb S^n}=-\Delta_{\mathbb S^n}+\tfrac{n(n-2)}{4}.\ee

Before we give a proof of this result we set up we set up the notation for the stereographic projection from South pole.  $\mathbb S^n$ is parameterized by coordinates $z'=(z_1,\ldots,z_n)$, $z_{n+1}$ such that $\abs{z'}^2+z_{n+1}^2=1$, and $\mathbb R^n$ by coordinates $x\in\mathbb R^n$.
Let $\phi:\mathbb S^n\backslash\{S\}\to\mathbb R^n$ be the map given by
$$x:=\phi(z',z_{n+1})=\frac{z'}{1+z_{n+1}}.$$
The push forward map is just
$$\phi^*\lp \frac{2}{1+\abs{x}^2}\rp=1+z_{n+1}$$
and, by conformality, it transforms the metric as
\begin{equation}\label{Jacobian}\phi^* h_{eq}=(1+z_{n+1})^{-2}h_{\mathbb S^{n}},\end{equation}
where $h_{eq}=\abs{dx}^2$ is the Euclidean metric.
For simplicity, we denote the conformal factor as
$$B^p f =(1+z_{n+1})^p f,\quad B_p f=2^p (1+\abs{x}^2)^{-p}f,$$
and note that the change of variable between them is simply
\begin{equation*}\label{relation}B^p \phi^*=\phi^* B_p,\end{equation*}
which will be used repeatedly in the following.

Let $-\Delta$ be the standard Laplacian on $\mathbb R^n$. It is related to the conformal Laplacian on the sphere \eqref{conformal-Laplacian-sphere} by the transformation law \eqref{conformal-property-Laplacian}, written as
\be\label{conformal-invariant-laplacian}P_1^{\mathbb S^n} B^{1-n/2}\phi^*=B^{-1-n/2}\phi^*(-\Delta).\ee
The conformal fractional Laplacian also satisfies the conformal covariance property \eqref{conformal-covariance}, which is
\be P^{\mathbb S^n}_ s  B^{ s -n/2}\phi^*=B^{- s -n/2}\phi^*(-\Delta)^ s ,\label{conformal-invariant-fractional}\ee
where $(-\Delta)^ s $ is the standard fractional Laplacian on $\mathbb R^n$ with respect to the Euclidean metric.\\

We show some preliminary commutator identities on $\mathbb R^n$:

\begin{lemma}\label{lemma-commutators}
Let $X=\sum x_i \partial_{x_i}$. Then
\begin{equation}\label{commutator1}\begin{split}
&[-\Delta,X]=2(-\Delta), \\
&[X,B_p]=-p\abs{x}^2 B_{p+1},\\
&[-\Delta,B_p]=p B_p (2X+n-(p-1)B_1\abs{x}^2)B_1,\\
&[(-\Delta)^ s , B_{-1}]= -s  \lp 2X+n+2( s -1)\rp(-\Delta)^{ s -1}.
\end{split}\end{equation}
\end{lemma}

\begin{proof}
The first three are direct calculations and can be found in \cite{Graham:conformal-powers}, while the last one is proved by Fourier transform. Compute
\begin{equation}\label{equation102}
[(-\Delta)^ s , B_{-1}]u(x)=\tfrac{1}{2}\left[(-\Delta)^s \{|x|^2u(x)\}-|x|^2(-\Delta)^su(x)\right].
\end{equation}
Fourier transform, with the multiplicative constants normalized to one, yields
 \begin{equation*}\begin{split}
\mathcal F\{|x|^2(-\Delta_x)^su(x)\}&=-\Delta_\xi\{|\xi|^{2s} \hat u(\xi)\}\\
&=-2s(n+2(s-1))|\xi|^{2(s-1)} \hat u(\xi)-4s|\xi|^{2(s-1)}\sum_{i=1}^n \xi_i \partial_{\xi_i} \hat u(\xi)-|\xi|^{2s} \Delta_{\xi} \hat u(\xi)\\
&=2s(n+2(s-1))|\xi|^{2(s-1)} \hat u(\xi)-4s\sum_{i=1}^n \partial_{\xi_i}\left\{\xi_i |\xi|^{2(s-1)} \hat u(\xi)\right\}-|\xi|^{2s} \Delta_{\xi} \hat u(\xi).
\end{split}
 \end{equation*}
Taking the inverse Fourier transform we obtain
\begin{equation*}
|x|^2(-\Delta_x)^su(x)=2s(n+2(s-1))(-\Delta_x)^{s-1}u(x)+4sXu(x)+(-\Delta_x)^s\{|x|^2 u(x)\}
\end{equation*}
which, in view of \eqref{equation102}, immediately yields the fourth identity in \eqref{commutator1}.\\
\end{proof}

\noindent
\textbf{Proof of theorem \ref{thm-formula-induction-sphere}: } By induction, it is clear that it is enough to show that
$$P_{1+ s }^{\mathbb S^n}=\lp P_1^{\mathbb S^n}+c_ s \rp P^{\mathbb S^n}_ s ,\quad \text{for}\quad c_ s =- s ( s +1).$$
Let $\mathcal P:=\lp P_1^{\mathbb S^n}+c_ s \rp P^{\mathbb S^n}_ s $. We claim that $\mathcal P$ is conformally covariant of order $1+ s $ in the sense of \eqref{conformal-covariance}, which, by uniqueness, will imply the proof of the theorem. Thus it is enough to show that $\mathcal P$ satisfies the conformal covariance identity
\be \mathcal P B^{ s +1-n/2} \phi^*=B^{- s -1-n/2}\phi^*(-\Delta)^{ s +1}.\label{conformal-invariant}\ee
For this, we first expand the left hand side of \eqref{conformal-invariant}. The idea is to use both the conformal invariance for the fractional Laplacian of exponent $ s $ \eqref{conformal-invariant-fractional} and for the standard conformal Laplacian \eqref{conformal-invariant-laplacian}, in order to relate the operator on $\mathbb S^n$ to the equivalent one on $\mathbb R^n$. We have
\bee\begin{split}
\text{(LHS)}&=\lp P_1^{\mathbb S^n}+c_ s \rp P^{\mathbb S^n}_ s  B^{ s +1-n/2} \phi^*\\
&=\lp P_1^{\mathbb S^n}+c_ s \rp P^{\mathbb S^n}_ s  B^{ s -n/2}\phi^* B_1\\
&=\lp P_1^{\mathbb S^n}+c_ s \rp B^{- s -n/2}\phi^*(-\Delta)^ s  B_1\\
&=\lp P_1^{\mathbb S^n}+c_ s \rp B^{1-n/2}\phi^* B_{-1- s }(-\Delta)^ s  B_1.
\end{split}\eee
Recalling again the conformal invariance for the conformal Laplacian $P^{\mathbb S^n}_1$ from \eqref{conformal-invariant-laplacian},
\be\label{formula100}\begin{split}
\text{(LHS)}&=\left[B^{-1-n/2}\phi^*(-\Delta)+c_ s  B^{1-n/2}\phi^*\right] B_{-1- s }(-\Delta)^ s  B_1\\
& =\phi^* B_{-1-n/2}\left[(-\Delta)+c_ s  B_2\right]B_{-1- s }(-\Delta)^ s  B_1.
\end{split}\ee
We next claim that
\be\label{claim1}\left[(-\Delta)+c_ s  B_2\right]B_{-1- s }(-\Delta)^ s  B_1=B_{- s } (-\Delta)^{ s +1},\ee
whose proof is presented below.
Therefore,  when we substitute the previous expression into \eqref{formula100} we obtain
\bee
\text{(LHS)} =\phi^* B_{-1-n/2- s } (-\Delta)^{ s +1},
\eee
which indeed implies \eqref{conformal-invariant} as we wished.\\

Now we give a proof for \eqref{claim1}. First note that
\begin{equation*}\begin{split}
(-\Delta)B_{-1- s }(-\Delta)^ s  B_1&=\left\{(-\Delta) B_{- s }\right\}B_{-1}(-\Delta)^ s  B_1 \\
&= B_{- s }(-\Delta)B_{-1}(-\Delta)^ s  B_1 +[(-\Delta), B_{- s }]B_{-1}(-\Delta)^ s  B_1
\end{split}\end{equation*}
and
\begin{equation*}\begin{split}
(-\Delta)B_{-1}(-\Delta)^ s   &=
(-\Delta)\left\{(-\Delta)^ s  B_{-1}+[B_{-1},(-\Delta)^ s ]\right\} \\
&= (-\Delta)^{1+ s } B_{-1}+(-\Delta)[B_{-1},(-\Delta)^ s ],
\end{split}\end{equation*}
so putting both expressions together yields
$$\left[(-\Delta)+c_ s  B_2\right]B_{-1- s }(-\Delta)^ s  B_1=B_{- s } (-\Delta)^{ s +1}+F_ s ,$$
where
$$F_ s :=B_{- s }(-\Delta)[B_{-1},(-\Delta)^ s ]B_1+[(-\Delta), B_{- s }]B_{-1}(-\Delta)^ s  B_1+c_ s  B_{1- s } (-\Delta)^ s  B_1.$$
A straightforward computation using the properties of the commutator from Lemma \ref{lemma-commutators} gives that
$F_ s \equiv 0$, and thus \eqref{claim1} is proved.

This concludes the proof of the Theorem.\\
\qed

From another point of view, on $\mathbb R^n$ with the Euclidean metric, the fractional Laplacian  for $s\in(0,1)$ can be computed as the principal value of the integral
\begin{equation}\label{fractional-Euclidean}(-\Delta)^ s  u(x)=C(n, s) \int_{\mathbb R^n} \frac{u(x)-u(\xi)}{\abs{x-\xi}^{n+2 s }}\;d\xi.\end{equation}
Our next objective is to give an analogous expression for $P_ s ^{\mathbb S^n}$ in terms of a singular integral operator, using stereographic projection in expression \eqref{fractional-Euclidean}:

\begin{prop}\label{prop-sphere-singular-integral}
Let $ s \in(0,1)$. Given $u(z)$ in $\mathcal C^\infty(\mathbb S^n)$, it holds
\begin{equation*}P_ s ^{\mathbb S^n} u(z)=\int_{\mathbb S^n} \left[u(z)-u(\zeta)\right] K_ s  (z,\zeta) \,d\zeta +A_{n, s }u(z),\end{equation*}
where the kernel $K_ s $ is given by
\begin{equation*}\label{kernel-sphere}K_ s (z,\zeta)=2^{s+n/2}C(n, s) \left(\frac{1-z_{n+1}}{1+z_{n+1}}\right)^{ s +n/2}
\left(\frac{1-\zeta_{n+1}}{1+\zeta_{n+1}}\right)^{ s +n/2} \frac{1}{(1-z\cdot \zeta)^{ s +n/2}}.
\end{equation*}
 and the (positive) constant
$$A_{n, s }=\frac{\Gamma\big(\frac{n}{2}+s\big)}{\Gamma\big(\frac{n}{2}-s\big)}.$$
\end{prop}

\begin{proof}
We recall the conformal covariance property for $P_s^{\mathbb S^n}$ from \eqref{conformal-covariance}
$$P^{\mathbb S^n}_ s  B^{ s -n/2}\phi^*=B^{- s -n/2}\phi^*(-\Delta)^s, $$
that for $u\in\mathcal C^{\infty}(\mathbb S^n)$ is equivalent to
\begin{equation*}\label{formula140}
P_ s ^{\mathbb S^n} u=B^{- s -n/2} \phi^*(-\Delta)^{ s } \left[ \phi_* B^{- s +n/2} u\right].\end{equation*}
From \eqref{fractional-Euclidean} we have
\begin{equation*}\label{formula130}
(-\Delta)^{ s } \left[  B_{- s +n/2} \phi_* u\right]=2^{- s +\frac{n}{2}}C(n, s) \int_{\mathbb R^n} \frac{(1+|x|^2)^{ s -n/2}u(\phi^{-1}(x))
-(1+|\xi|^2)^{ s -n/2}u(\phi^{-1}(\xi))}{\abs{x-\xi}^{n+2 s }}\;d\xi.\end{equation*}
We pull back to $\mathbb S^n$, with coordinates
$$x=\phi(z),\quad \zeta=\phi(\xi),$$
recalling the Jacobian of the transformation from \eqref{Jacobian}. Also
note that
$$\abs{x-\xi}^2=2\frac{1-z\cdot \zeta}{(1-z_{n+1})(1-\zeta_{n+1})}.$$
Therefore
\begin{equation*}\begin{split}
P_ s ^{\mathbb S^n} u(z)=C(n, s)2^{s+n/2} (1+z_{n+1})^{- s -n/2}\int_{\mathbb S^n} &\left[(1+z_{n+1})^{- s +\frac{n}{2}}u(z)
-(1+\zeta_{n+1})^{- s +\frac{n}{2}}u(\zeta)\right]\\
&\cdot\frac{(1-z_{n+1})^{ s +n/2}(1-\zeta_{n+1})^{ s +n/2}}{(1-z\cdot \zeta)^{ s +n/2}}(1+\zeta_{n+1})^{-n}\;d\zeta.
\end{split}\end{equation*}
Writing
$$u(z)=u(z)\frac{(1+\zeta_{n+1})^{- s +n/2}}{(1+z_{n+1})^{- s +n/2}}
+u(z)\left[1-\frac{(1+\zeta_{n+1})^{- s +n/2}}{(1+z_{n+1})^{- s +n/2}}\right],$$
we can arrive at
\begin{equation}\label{Pgamma1}P_ s ^{\mathbb S^n} u(z)=\int_{\mathbb S^n} \left[u(z)-u(\zeta)\right] K_ s  (z,\zeta) \,d\zeta +u(z) \tilde K_ s (z),\end{equation}
for
\begin{equation*}
\tilde K_s(z)=2^{\frac{n}{2}+s}C(n,s)\int_{\mathbb S^n} \frac{(1+z_{n+1})^{-s+\frac{n}{2}}-(1+\zeta_{n+1})^{-s+\frac{n}{2}}}{(1+z\cdot \zeta)^{s+\frac{n}{2}}}\frac{(1-z_{n+1})^{s+\frac{n}{2}}}{(1+z_{n+1})^{n}}
\frac{(1-\zeta_{n+1})^{s+\frac{n}{2}}}{(1+\zeta_{n+1})^{n}}\,d\zeta.
\end{equation*}
On the other hand, it is possible to show that $\tilde K_ s $ is constant in $z$. We have not attempted a direct proof; instead, we compare  \eqref{curvature-sphere} and \eqref{Pgamma1}  applied to $u\equiv 1$. As a consequence,
\begin{equation*}
\tilde K_s(z)\equiv\frac{\Gamma\big(\frac{n}{2}+s\big)}{\Gamma\big(\frac{n}{2}-s\big)}.
\end{equation*}
This yields the proof of the Proposition.
\end{proof}


\section{The conformal fractional Laplacian on the cylinder}\label{section:cylinder}

 Up to now, we have just considered conformally compact manifolds, for which the  conformal infinity $(M,[h])$ is compact. But one could also look at the non-compact case. This is, perhaps, one of the most interesting issues since the definition on of the fractional conformal Laplacian, being a non-local operator, is not clear when $M$ has singularities. In this section we consider the particular case when $M$ is a cylinder.

 Let $M=\mathbb R^n\backslash \{0\}$ with the cylindrical metric given by $$h_0:=\frac{1}{r^{2}}|dx|^2$$ for $r=|x|$. Use the Emden-Fowler change of variable $r=e^{-t}$, $t\in\mathbb R$, and remark that the Euclidean metric may be written as
 \begin{equation}\label{conformal-change-cylinder}|dx|^2=dr^2+r^2 h_{\mathbb S^{n-1}}=e^{-2t}[dt^2+h_{\mathbb S^{n-1}}]=:e^{-2t}h_0.
\end{equation}
Thus, in these new coordinates, $M$  may be identified with  the manifold $\mathbb R\times \mathbb S^{n-1}$ with the metric $h_0=dt^2+h_{\mathbb S^{n-1}}$.

The conformal covariance property \eqref{conformal-covariance} allows to formally write the conformal fractional Laplacian on the cylinder from the standard fractional Laplacian on Euclidean space. Indeed,
\begin{equation*}
P_s^{h_0}(v)=r^{\frac{n+2s}{2}} P^{|dx|^2}_{s}(r^{-\frac{n-2s}{2}} v)=r^{\frac{n+2s}{2}}
(-\Delta)^{s}u,
\end{equation*}
where we have set
\begin{equation}\label{uv}u=r^{-\frac{n-2s}{2}}v.\end{equation}
This relation also allows to calculate the fractional curvature of a cylinder. It is the (positive) constant
 \begin{equation}\label{constant-cylinder}
c_{n,s}:=Q_s^{h_0}=P_s^{h_0}(1)=r^{\frac{n+2s}{2}} (-\Delta)^s(r^{-\frac{n-2s}{2}}) =2^{2s}\left(\frac{\Gamma(\frac{1}{2}(\frac{n}{2}+s))}
{\Gamma(\frac{1}{2}(\frac{n}{2}-s))}\right)^2,\end{equation}
where the last equality is shown taking into account the Fourier transform of a homogeneous distribution.

 In \cite{DelaTorre-Gonzalez} the authors compute the principal symbol of the operator $P_s^{h_0}$ on $\mathbb R\times \mathbb S^{n-1}$ using the spherical harmonic decomposition for $\mathbb S^{n-1}$. This proof is close in spirit to the calculation we presented in the previous section for the sphere case, once we understand the underlying geometry. In fact, the standard cylinder $(\mathbb R \times \mathbb S^{n-1},h_0)$ is the conformal infinity of the Riemannian AdS space, which is another simple example of a Poincar\'e-Einstein manifold. AdS space may be described as the $(n+1)$-dimensional manifold with metric
\begin{equation*}
{g^+}=\rho^{-2}\lp d\rho^2+\left(1+\tfrac{\rho^2}{4}\right)^2dt^2+\left(1-\tfrac{\rho^2}{4}\right)^2 h_{\mathbb S^{n-1}}\rp,\end{equation*}
where  $\rho\in(0,2]$ and $t\in\mathbb R$. As in the sphere case, the calculation of the Fourier symbol of $P_s^{h_0}$ goes by reducing the scattering problem \eqref{equation-GZ}-\eqref{general-solution} to an ODE in the variable $\rho$ and then looking at its asymptotic behavior at $\rho=0$ and $\rho=2$. We will not present the proof of Theorem \ref{thm-symbol} below but refer to the original paper \cite{DelaTorre-Gonzalez}, since the new difficulties are of technical nature only.

 With some abuse of notation, let $\mu_m=m(m+n-2)$, $m=0,1,2,...$ be the eigenvalues of $-\Delta_{\mathbb S^{n-1}}$, repeated according to multiplicity. Then, any function $v$ on $\mathbb R\times \mathbb S^{n-1}$ may be decomposed as $\sum_{m} v_m(t) E_m$, where $\{E_m\}$ is a basis of eigenfunctions. Let
 \begin{equation*}\label{fourier}
\hat{v}(\xi)=\frac{1}{\sqrt{2\pi}}\int_{\mathbb R}e^{-i\xi \cdot t} v(t)\,dt
\end{equation*}
be our normalization for the one-dimensional Fourier transform. Then the operator $P_s^{h_0}$ diagonalizes under such eigenspace decomposition, and moreover, it is possible to calculate the Fourier symbol of each projection. More precisely:

\begin{thm}[\cite{DelaTorre-Gonzalez}]\label{thm-symbol}
 Fix $s\in (0,\tfrac{n}{2})$ and let $P^m_{s}$ be the projection of the operator $P^{h_0}_s$ over each eigenspace $\langle E_m\rangle$. Then
$$\widehat{P_s^m (v_m)}=\Theta^m_s(\xi) \,\widehat{v_m},$$
and this Fourier symbol is given by
\begin{equation}\label{symbol-cylinder}
\Theta^m_{s}(\xi)=2^{2s}\frac{\Big|\Gamma\big(\tfrac{1}{2}+\tfrac{s}{2}
+\tfrac{\sqrt{(\tfrac{n}{2}-1)^2+\mu_m}}{2}+\tfrac{\xi}{2}i\big)\Big|^2}
{\Big|\Gamma\big(\tfrac{1}{2}-\tfrac{s}{2}+\tfrac{\sqrt{(\tfrac{n}{2}-1)^2+\mu_m}}{2}
+\tfrac{\xi}{2}i\big)\Big|^2}.
\end{equation}
\end{thm}

Let us restrict to the space of radial functions $v=v(t)$, which corresponds to the eigenspace with $m=0$, and denote $\mathscr L_s:=P^0_{s}$. Then the Fourier symbol of $\mathscr L_s$ is given by
\begin{equation*}\label{symbolk0}
\Theta^0_{s}(\xi)=2^{2s}\frac{\left|\Gamma(\tfrac{n}{4}+\tfrac{s}{2}+\tfrac{\xi}{2}i)\right|^2}
{\left|\Gamma(\tfrac{n}{4}-\tfrac{s}{2}+\tfrac{\xi}{2}i)\right|^2}.
\end{equation*}

Again, in parallel to Proposition \ref{prop-sphere-singular-integral} in the sphere case, it is possible to give a singular integral formulation for the pseudodifferential operator $\mathscr L_s$ acting on $\langle E_0\rangle$:

\begin{prop}[\cite{DelaTorre-DelPino-Gonzalez-Wei}] Given $v=v(t)$ smooth, $t\in\mathbb R$, we have for $\mathscr L_s$:
\begin{equation}\label{mathcalL}
 \mathscr{L}_{s}v(t)=C(n,s)P.V.   \int_{-\infty}^\infty
 (v(t)-v(\tau))K(t-\tau)\,d\tau + c_{n,s} v(t),
\end{equation}
for the kernel
\begin{equation*}
K({\xi})= c_n(\sinh {\xi})^{-1-2s} (\cosh {\xi})^{ \tfrac{2-n+2s}{2}} \mbox{ }_2 F_1 \left(\tfrac{A+1}{2}-B, \tfrac{A}{2} -B +1; A-B +1; (\textnormal{sech } {\xi})^2\right).
\end{equation*}
where $A= \tfrac{n+2s}{2}$, $B =1 +s$,
$c_n$ is a dimensional constant and the value of $c_{n,s}$ is given in \eqref{constant-cylinder}.\\
\end{prop}

It can be shown (\cite{DelaTorre-DelPino-Gonzalez-Wei}) that the asymptotic behavior of this kernel $K$ is
\begin{eqnarray*}
 K({\xi})  &\sim&  |{\xi}|^{-1-2\gamma}\quad \mbox{as}\quad |\xi|\rightarrow 0,\\
K(\xi)&\sim& e^{-|\xi|\tfrac{n+2\gamma}{2}} \quad \mbox{as}\quad |\xi|\rightarrow \infty.
 \end{eqnarray*}
This shows that, at the origin, its singular behavior corresponds to that of the one-dimensional fractional Laplacian but, at infinity, it has a much faster decay. Levy processes arising from this type of generators are known as tempered stable processes; they combine both $\alpha$-stable (in the short range) and Gaussian (in the long range) trends. \\


Before we continue with this exposition let us mention  a related problem:  to construct solutions to the  fractional Yamabe problem on $\mathbb R^n$, $s\in(0,1)$, with an isolated singularity at the origin.  This means that one seeks positive solutions of
\begin{equation}\label{equation0}(-\Delta)^{s}u=c_{n,s}u^{\frac{n+2s}{n-2s}}\text{ in }\mathbb R^n \setminus \{0\},\end{equation}
where $c_{n,s}$ is any positive constant that will be normalized as in \eqref{constant-cylinder}, and such that
$$u(x)\to \infty\quad\text{as}\quad |x|\to 0.$$
For technical reasons, one needs to assume here that $n>2+2s$. Because of the well known extension theorem for the fractional Laplacian \eqref{equation0} is equivalent to the boundary reaction problem
\begin{equation}\label{equation1}\left\{
\begin{split}
-\divergence(y^{a}\nabla U)=0&\text{ in } \mathbb R^{n+1}_+,\\
U=u&\text{ on }\mathbb R^n\setminus\{0\},\\
-{d}^*_{s}\lim_{y\rightarrow 0}y^a\partial_y u=c_{n,s}u^{\frac{n+2s}{n-2s}}&\text{ on }\mathbb R^n\setminus\{0\}.
\end{split}\right.
\end{equation}

Our model for an isolated singularity is the cylindrical solution, given by $U_1=P *_x u_1$ with $u_1(r)=r^{-\frac{n-2s}{2}}$. In the recent paper \cite{CaffarelliJinSireXiong} the authors characterize all the nonnegative solutions to \eqref{equation1}. Indeed, if the origin is not a removable singularity, then $u(x)$ is radial in the $x$ variable and, if $u=U(\cdot,0)$, then near the origin one must have that
\begin{equation*}\label{asymptotics}
c_1r^{-\tfrac{n-2s}{2}}\leq u(x)\leq c_2r^{-\tfrac{n-2s}{2}},
\end{equation*}
where $c_1$, $c_2$ are positive constants.

Positive radial solutions for \eqref{equation0} have been studied in the papers \cite{DelaTorre-Gonzalez,DelaTorre-DelPino-Gonzalez-Wei}. It is clear from the above that one should look for solutions of the form \eqref{uv}
 for some  function $v=v(r)$ satisfying $0<c_1\leq v\leq c_2$. In the classical case $s=1$, equation \eqref{equation0} reduces to a standard second order ODE. However, in the fractional case it becomes a fractional order ODE and, as a consequence, classical methods cannot be directly applied here.

 One possible point of view is to rewrite
 \eqref{equation0} in the new metric $h_0$. Since the metrics $|dx|^2$ and $h_0$ are conformally related by \eqref{conformal-change-cylinder}, we prefer to use $h_0$ as a background metric and thus any conformal change may be rewritten as
$$h_v=u^{\frac{4}{n-2s}}|dx|^2=v^{\frac{4}{n-2s}}h_0,$$
 where $u$ and $v$ are related by \eqref{uv}. Then the original problem \eqref{equation0} is equivalent to the fractional Yamabe problem on $\mathbb R\times \mathbb S^{n-1}$: fixed $h_0$ as a background metric on $\mathbb R\times \mathbb S^{n-1}$, find a conformal metric $h_v$ of positive constant fractional curvature $Q^{h_v}_s$, i.e., find a positive smooth solution $v$ for
\begin{equation}\label{equation2}P_s^{h_0}(v)=c_{n,s} v^{\frac{n+2s}{n-2s}}\quad \text{on}\quad \mathbb R\times \mathbb S^{n-1}.\end{equation}
A complete study of radial solutions
$v=v(t)$, $0<c_1\leq v\leq c_2$, for this equation is not available  since is not an ODE. The local case $s=1$, however, is well known since it reduces to understanding the phase-portrait of a Hamiltonian ODE (see the lecture notes \cite{Schoen:notas}, for instance), and periodic solutions constructed in this way are known as Delaunay solutions for the scalar curvature.

 Fractional Delaunay solutions $v_L$ to equation \eqref{equation2}, i.e., radially symmetric periodic solutions in the variable  $t\in\mathbb R$ for a given period $L$, are constructed in \cite{DelaTorre-DelPino-Gonzalez-Wei} using a variational method that we sketch here: if $v$ is radial, then \eqref{equation2} is equivalent to
$$\mathscr{L}_{s} v=c_{n,s}v^{\frac{n+2s}{n-2s}}, \quad t\in\mathbb R,$$
where $\mathscr L_s$ is given in \eqref{mathcalL}. For periodic functions $v(t+L)=v(t)$ it can be rewritten as
\begin{equation}\label{equation-periodic}
\mathscr{L}_{s}^L v=c_{n,s}v^{\frac{n+2s}{n-2s}}, \quad t\in\mathbb [0,L],
\end{equation}
for
\begin{equation*}
 \mathscr{L}_{s}^L v(t)=C(n,s)P.V.   \int_{0}^L
 (v(t)-v(\tau))K_L(t-\tau)\,d\tau + c_{n,s} v(t),
\end{equation*}
and $K_L$ is a periodic singular kernel given by $K_L (\xi)= \sum_{j\in\mathbb Z}K(\xi- jL)$.\\

We find a bifurcation behavior at the value of $L$, denoted by $L_0$, where the first eigenvalue for the linearization of problem \eqref{equation-periodic} crosses zero. Moreover, for each period $L> L_0$ there exists a periodic solution $v_L$:

 \begin{thm}[\cite{DelaTorre-DelPino-Gonzalez-Wei}]
Consider the variational formulation for equation \eqref{equation-periodic},
written as
$$b(L)=\inf \left\{\mathcal{F}_L(v)\,:\, v\in H^s(0,L), \, v \,\text{is}\,L\text{-periodic}\right\}$$
where
\begin{equation*}
\mathcal{F}_L(v) =\frac{ C(n,s)\int_0^L \int_0^L (v(t)-v(\tau))^2 K_L (t-\tau) \,dt \,d\tau +c_{n,s} \int_0^L v(t)^2 \,dt}{ \big(\int_0^L v(t)^{2*} \,dt\big)^{2/{2^*}}}.
\end{equation*}
Then there is a unique $L_0>0$ such that $b(L)$ is attained by a nonconstant (positive) minimizer  $v_L$ when $L>L_0$ and when $L \leq L_0$  $b(L)$ is attained by the constant only.
\end{thm}

Such $v_L$ for $L>L_0$ are known as the Delaunay solutions for the fractional curvature, and they can be characterized almost explicitly. We remark that, geometrically, the constant solution $v_{L_0}$  corresponds to the standard cylinder, while $v_L\to v_\infty$ as $L\to \infty$, where
 $v_\infty(t)=c(\cosh (t))^{-\frac{n-2s}{2}}$ corresponds to a standard sphere (i.e., the bubble solution \eqref{bubbles}, normalized accordingly). For other values of $L$ we have a characterization as a bubble tower; in fact:

\begin{prop}[\cite{Ao-DelaTorre-Gonzalez-Wei}]
We have that
\begin{equation}\label{bubble-tower}v_L=\sum_{j\in \mathbb Z}v_\infty(\cdot-jL)+\phi_L,$$ where
$$\|\phi_L\|_{H^s(0,L)}\to 0\quad\text{as}\quad L\to \infty.\end{equation}
Moreover, for $L$ large we have the following Holder estimates on $\phi_L$:
\begin{equation*}
\|\phi_L\|_{\mathcal C^{\alpha}([0,L])}\leq Ce^{-\frac{(n-2s)L}{4}(1+\xi)}
\end{equation*}
for some $\alpha\in (0,1)$ and $\xi>0$ independent of $L$.
\end{prop}

Finally, Delaunay-type solutions can be used in gluing problems since they model isolated singularities. In \cite{Ao-DelaTorre-Gonzalez-Wei} the authors modify the methods in \cite{Schoen:singular-solutions,Mazzeo-Pacard:isolated} to construct complete metrics on the sphere of constant scalar curvature with a finite number of isolated singularities:

\begin{thm}[\cite{Ao-DelaTorre-Gonzalez-Wei}]
Let $\Lambda=\{p_1,\ldots,p_N\}$ be a set of $N$ points in $\mathbb R^n$. There exists a smooth positive solution $u$ for the problem
\begin{equation}\label{singular-Yamabe}
\left\{\begin{split}
&(-\Delta)^{s}u=u^{\frac{n+2s}{n-2s}}\quad\text{in }\mathbb R^n \setminus \Lambda,\\
&u(x)\to \infty \quad \text{as }x\to \Lambda.
\end{split}\right.
\end{equation}
\end{thm}

The proof of this theorem consists of a Lyapunov-Schmidt reduction involving a different perturbation of each bubble in the bubble tower \eqref{bubble-tower}, in the spirit of \cite{Malchiodi}. Note that the compatibility conditions that arise come from an infinite-dimensional Toda-type system; in addition, they do not impose any restriction on the location of the singular points, only on the neck sizes $L$ at each singularity.

\section{The non-compact case}\label{section-noncompact}

Once the isolated singularities case has been reasonably well understood, we turn to the study of higher dimensional singularities. From the analysis point of view, one wishes to understand the semilinear problem
\begin{equation}\label{singular-Yamabe}
\left\{\begin{split}
&(-\Delta)^{s}u=c\,u^{\frac{n+2s}{n-2s}}\text{ in }\mathbb R^n \setminus \Lambda,\\
&u(x)\to \infty \text{ as }x\to \Lambda,
\end{split}\right.
\end{equation}
where $\Lambda$ is a closed set of Hausdorff dimension $0<k<n$ and $c\in\mathbb R$. The first difficulty one encounters is precisely how to define the fractional Laplacian $(-\Delta)^s$ on $\Omega:=\mathbb R^n\setminus \Lambda$ since it is a non-local operator. Nevertheless, as in the cylinder case, this is better understood from the conformal geometry point of view. \\

In order to put \eqref{singular-Yamabe} into a broader context, let us give a brief review of the classical singular Yamabe problem ($s=1$).
Let $(M,h)$ be a compact $n$-dimensional Riemannian manifold, $n \geq 3$, and $\Lambda \subset M$ is any
closed set as above. We are concerned with the existence and geometric properties of
complete (non-compact) metrics of the form $h_u = u^{\frac{4}{n-2}}h$ with constant scalar curvature. This corresponds to solving
the partial differential equation (recall \eqref{Yamabe-equation})
\begin{equation*}
-\Delta_{h} u + \tfrac{n-2}{4(n-1)} R_{h} u = \tfrac{n-2}{4(n-1)}R\, u^{\frac{n+2}{n-2} }, \quad u > 0,
\end{equation*}
where $ R_{h_u}\equiv R$ is constant and with a boundary condition that $u \to \infty$ sufficiently quickly at $\Lambda$ so that
$h_u$ is complete. It is known
that solutions with $R< 0$ exist quite generally if $\Lambda$ is large in a capacitary
sense (\cite{Loewner-Nirenberg,Lab}), whereas for $R > 0$ existence is only known when $\Lambda$ is a smooth
submanifold (possibly with boundary) of dimension $k < (n-2)/2$ (\cite{MP,F}).

There are both analytic and geometric motivations for studying this problem. For example, in the positive case ($R > 0$), solutions
to this problem are actually weak solutions across the singular set (\cite{SY}), so these results fit into the broader investigation of
possible singular sets of weak solutions of semilinear elliptic equations.

On the geometric side, a well-known theorem by Schoen and Yau (\cite{SY,Schoen-Yau:book}) states that if $(M,h)$ is a compact manifold with a locally conformally flat metric $h$ of
positive scalar curvature, then the developing map $D$ from the universal cover $\widetilde{M}$ to $\mathbb S^n$, which
by definition is conformal, is injective, and moreover, $\Lambda := \mathbb S^n \setminus D(\widetilde{M})$ has Hausdorff
dimension less than or equal to $(n-2)/2$. Regarding the lifted metric $\tilde{h}$ on $\widetilde{M}$ as a metric
on $\Omega$, this provides an interesting class of solutions of the singular Yamabe problem which are periodic with respect to a
Kleinian group, and for which the singular set $\Lambda$ is typically nonrectifiable. More generally, they also
show that if $h_{\mathbb S^n}$ is the canonical metric on the sphere and if $h= u^{\frac{4}{n-2}}h_{\mathbb S^n}$ is a complete metric
with positive scalar curvature and bounded Ricci curvature on a domain $\Omega = \mathbb S^n \setminus \Lambda$, then
$\dim \Lambda \leq (n-2)/2$.\\

Going back to the non-local case, although it is not at all clear how to define $P_s^{\tilde h}$ and $Q_s^{\tilde h}$ on a general complete (non-compact) manifold $(\Omega,\tilde h)$, in the paper \cite{Gonzalez-Mazzeo-Sire} the authors  give a reasonable definition when $\Omega$ is an open dense set in a compact manifold $M$ and the metric $\tilde h$ is conformally related to a smooth metric $h$ on $M$. Namely, one can define them by demanding that the
conformal property \eqref{conformal-covariance} holds (as usual,
we assume that a Poincar\'e-Einstein filling $(X,g^+)$ has been fixed). Note, however, that this is not as simple as it first appears since, because of
the nonlocal character of $P_s^{\tilde h}$, we must extend $u$ as a distribution on all of $M$.  There is no difficulty
in using the relationship \eqref{conformal-covariance} to define $P_s^{\tilde h} \phi$ when $\phi \in \mathcal C^\infty_0(\Omega)$.
From here one can use an abstract functional analytic argument to extend $P_s^{\tilde h}$ to act on any $\phi \in L^2(\Omega, dv_{\tilde h})$.
Indeed, it is straightforward to check that the operator $P_s^{\tilde h}$ defined in this way is essentially self-adjoint on
$L^2(\Omega, dv_{\tilde h})$ when $s$ is real. However, observe that $P_s^{\tilde h} = (-\Delta_{\tilde h})^{s} + \mathcal K$, where
$\mathcal K$ is a pseudo-differential operator of order $2s-1$. Furthermore, $(-\Delta_{\tilde h})^{s}$ is self-adjoint. Since $\mathcal K$ is a lower order symmetric perturbation, then $P_s^{\tilde h}$ is also essentially self-adjoint.

Another interesting development is \cite{Guillarmou-Qing}, where they give a sharp spectral characterization of  Poincar\'e-Einstein manifolds with conformal infinity of positive Yamabe type.\\

The singular fractional Yamabe problem on $(M,[h])$ is then formulated as
\begin{equation}\label{singular-Yamabe1}
\left\{\begin{split}
&P^{h}_{s}u=c u^{\frac{n+2s}{n-2s}}\quad\text{in }M \setminus \Lambda,\\
&u(x)\to \infty \quad\text{as }x\to \Lambda,
\end{split}\right.
\end{equation}
for $c\equiv Q_s^{\tilde h}$ constant.
A separate, but also very interesting issue, is whether $c>0$ implies that the conformal factor
$u$ is actually a weak solution of \eqref{singular-Yamabe1} on all of $M$.\\

The first result in \cite{Gonzalez-Mazzeo-Sire} partially generalizes Schoen-Yau's theorem:

\begin{thm}[\cite{Gonzalez-Mazzeo-Sire}]
Suppose that $(M^n,h)$ is compact and $h_u = u^{\frac{4}{n-2s}}h$ is a complete metric on $\Omega = M \setminus
\Lambda$,  where $\Lambda$ is a smooth $k$-dimensional submanifold in $M$.  Assume furthermore that $u$ is polyhomogeneous
along $\Lambda$ with leading exponent $-n/2 + s$. If $s\in\left(0,\frac{n}{2}\right)$, and if $Q_s^h > 0$ everywhere
for any choice of asymptotically Poincar\'e-Einstein extension $(X,g^+)$  then $n$, $k$ and $s $ are restricted by the inequality
\begin{equation}
\Gamma\lp\frac{n}{4} - \frac{k}{2} + \frac{s}{2}\rp \Big/ \Gamma\lp\frac{n}{4} - \frac{k}{2} - \frac{s}{2}\rp > 0.
\label{eq:dimrest}
\end{equation}
This inequality holds in particular when
\begin{equation}\label{dim}
k < \frac{n-2s}{2},\end{equation}
and in this case then there is a unique distributional extension of $u$ on all of $M$ which is still a solution of \eqref{singular-Yamabe1} on all of $M$.
\end{thm}

As we have noted, inequality \eqref{eq:dimrest} is satisfied whenever $k < (n-2s)/2$, and in fact is equivalent to this simpler
inequality when $s  = 1$.  When $s =2$, i.e. for the standard $Q-$curvature, this result is already known: \cite{Chang-Hang-Yang} shows that complete metrics with $Q_2 > 0$ and positive scalar curvature must have singular set with dimension less than $(n-4)/2$, which again agrees with \eqref{eq:dimrest}.

Of course, the main open question is to remove the smoothness assumption on the singular set $\Lambda$. Recent results of \cite{Zhang} show that, under a positive scalar curvature assumption, if $Q_s>0$ for $s\in(1,2)$, then \eqref{dim} holds for any $\Lambda$.

We also remark that a dimension estimate of the type \eqref{eq:dimrest} implies some topological restrictions on $M$: on the homotopy (\cite{Schoen-Yau:book}, chapter VI), on the cohomology (\cite{Nayatani}), or even classification results in the low dimensional case (\cite{Izeki}).\\

Finally, one can also obtain existence of solutions when $s$ is sufficiently near $1$ and $\Lambda$ is smooth by perturbation theory:

\begin{thm}[\cite{Gonzalez-Mazzeo-Sire}]
Let $(M^n, h)$ be compact with nonnegative Yamabe constant and $\Lambda$ a $k$-dimensional submanifold
with $k < \frac12 (n-2)$. Then there exists an $\epsilon > 0$ such that if $s \in (1-\epsilon, 1 + \epsilon)$,
there exists a solution to the fractional singular Yamabe problem \eqref{singular-Yamabe1} with $c > 0$ which is complete
on $M \setminus \Lambda$.
\end{thm}

The next step is to show existence for any value of $s\in(0,1)$. If we do not worry about completeness, partial results have been obtained in \cite{Ao-Chan-Gonzalez-Wei}. We hope to return to this problem elsewhere.

\section{Uniqueness}\label{section:uniqueness}

One of the main questions that arises is, given a manifold $(M^n,h)$, is there a canonical way to define the conformal fractional Laplacian on $M$? this question is equivalent to ask how many Poincar\'e-Einstein fillings $(X^{n+1},g^+)$ one can find. The answer is, in general, not unique, unless the conformal infinity is the round sphere (or equivalently, $\mathbb R^n$) (see the survey \cite{Chang-Qing-Yang:notes}, for instance).

In this section we would like to describe two Poincar\'e-Einstein fillings on the topologically same 4-manifold with the same conformal infinity. This construction comes from the study of thermodynamics of black holes in Anti-de Sitter Space \cite{Hawking-Page}, and it is  explained  in the survey \cite{Chang-Qing-Yang:notes}, but we repeat it here for completeness.

The AdS-Schwarzchild manifold is an Einstein 4-manifold described as  $\mathbb R^2\times\mathbb S^2$ with the metric
\begin{equation*}\label{metrica}
g_{m}^+=Vdt^2+V^{-1}dr^2+r^2h_{\mathbb S^2}\quad\mbox{for}\quad
V=1+r^2-\frac{2m}{r}.
\end{equation*}
We call $r_m$ the positive root for $1+r^2-\frac{2m}{r}=0$, and we restrict $r\in [r_m,+\infty)$. $m>0$ is known as the mass parameter, to be chosen later.

Even though this metric seems singular ar $r_m$, we will prove that this is not the case if we make the $t$ variable periodic, i.e., $t\in \mathbb S^1(L)$. To see this, define a function
$\rho:(r_{h},\infty)\to(0,\infty)$ by $$\rho(r)=\int_{r_m}^r V^{-1/2}.$$
One can check that for $r$ near  $r_m$,
$$g_{+1}^m\sim d\rho^2+\frac{(V'(r_m))^2}{4}\rho^2dt^2+r^2{h_{\mathbb S^2}},$$
so the singularity at $r=r_m$ is of the same type as the origin in standard polar coordinates. Thus thus we need to make the $t$ variable periodic, i.e, $0\leq t \leq 2\pi L,$ for
\begin{equation}\label{L}L:=L(m)=\frac{V'(r_m)}{2}=\frac{2r_m}{3r_m^2+1}.
\end{equation}

To show that $g_m^+$ is conformally compact,  we change to the defining function  $s=\frac{1}{r}$. Since $V(r)\approx \frac{1}{s^2}$ when $s\to 0$, then
$$g_{m}^+\sim \frac{1}{s^2}[ds^2+dt^2+{h_{\mathbb S^2}}]\quad\text{as} \quad s\to 0.$$
Therefore, for each $m>0$, $g_m^+$ is Poincar\'e-Einstein and its  conformal infinity is $\mathbb S^1(L)\times \mathbb S^{2}$ with the metric $h_m:=dt^2+h_{\mathbb S^2}$.

But we could ask the reverse question of, given $L$, how many Poincar\'e-Einstein fillings one can find for $S^1(L)\times \mathbb S^{2}$. Looking at \eqref{L}, $r_{m}=\frac{1}{\sqrt{3}}$  is a critical point for $L(r_m)$, actually a maximum with value $$L\big(\tfrac{1}{\sqrt 3}\big)=\tfrac{1}{\sqrt{3}}.$$
There holds \begin{itemize}
  \item For  $0<L<1/\sqrt 3$, we can find two different masses $m_1$ and $m_2$ with the same $L(m)$.
  \item For $L=1/\sqrt 3$, there exist only one mass $m$ which gives $L(m)$.
  \item If $L>1/\sqrt 3$, there does not exist any mass which gives $L(m)$.
\end{itemize}
Thus for the same conformal infinity $\mathbb S^1(L)\times\mathbb S^2$, when $0<L<\frac{1}{\sqrt{3}}$ there  are two non-isometric AdS-Schwarzschild spaces with metrics $g^+_{m_1}$ and $g^+_{m_2}$. The  natural question now is to calculate the symbol of the conformal fractional Laplacian $P_s^m$ on the conformal infinity for each model. This calculation is similar to that of \eqref{intertwining-sphere} for the sphere and \eqref{symbol-cylinder} for the cylinder. But, unfortunately, the spherical harmonic decomposition yields a more complicated ODE that we have not been able to solve analytically.

\section{An introduction to hypersurface conformal geometry}\label{section:hypersurface}

Let  $(\overline X^{n+1},\bar g)$ be any smooth compact manifold with boundary $(M^n, h)$, where $h=\bar g|_M$, for instance, a domain in $\mathbb R^{n+1}$ with the Euclidean metric. One would like to understand the conformal geometry of  $M$ as an embedded hypersurface with respect to the given filling metric $\bar g$, and to produce new extrinsic conformal invariants on this hypersurface.

In this discussion we are mostly interested in the construction of non-local objects on $M$, in particular, the conformal fractional Laplacian, and to understand how this new $P^h_s$ depends on the geometry of the background metric $\bar g$. A  good starting reference is the recent paper  \cite{Graham:singular-Yamabe}, although the author is more interested in renormalized volume rather than scattering (see also the parallel development by \cite{Gover-Waldron:conformal-hypersurface-geometry,Gover-Waldron:renormalized-volume} in the language of tractor calculus).

Let $\rho$ be a geodesic defining function for $M$. This means, in particular, that
$\bar g=d\rho^2+h_\rho$, where $h_\rho$ is a one parameter family of metrics on $M$  with $h_\rho|_{\rho=0}=h$. We would like to produce a suitable asymptotically hyperbolic filling metric $g^+$ for which the scattering problem  \eqref{equation-GZ}-\eqref{general-solution} can be solved in terms of information from $\bar g$ only. Looking at \eqref{E-general}, the reasonable assumption is to ask that $g^+$ has constant scalar curvature
\begin{equation}\label{constant-scalar}R_{g^+}=-n(n+1).\end{equation}
Thus we seek a new defining function $\hat \rho=\hat \rho(\rho)$ such that if we define
$$g^+=\frac{\bar g}{\hat \rho^2},$$
then this $g^+$ is asymptotically hyperbolic and satisfies \eqref{constant-scalar}.  Remark that sometimes it will be more convenient to write $g^+=u^{\frac{4}{n-1}}\bar g$ for $u=\hat \rho^{-\frac{n-1}{2}}$.

This problem for $g^+$ reduces to the singular Yamabe problem of Loewner-Nirenberg (\cite{Loewner-Nirenberg}) for constant negative scalar curvature and it has been well studied (\cite{Mazzeo:regularity-singular-Yamabe,Aviles-McOwen,Andersson-Chrusciel-Friedrich}, for instance). In PDE language, looking at the conformal transformation law for the usual conformal Laplacian \eqref{Yamabe-equation}, it amounts to find a positive solution $u$ in $X$ to the equation
$$-\Delta_{\bar g} u+\tfrac{n-1}{4n} R_{\bar g} u=-\tfrac{n^2-1}{4}u^{\frac{n+3}{n-1}}$$
that has the asymptotic behavior
$$u\sim \rho^{-\frac{n-1}{2}}\quad\text{near }\partial X$$
(recall that we are working on an $(n+1)$-dimensional manifold).
It has been shown that such solution exists and it has a very specific polyhomogeneous expansion near $\partial X$, so that
\begin{equation*}
g^+=\frac{\bar g(1+\rho \alpha+\rho^{n+1}\beta)}{\rho^2},
\end{equation*}
where $\alpha\in \mathcal C^{\infty}(\overline X)$ and $\beta\in \mathcal C^\infty(X)$ has a polyhomogeneous expansion with log terms. This type of expansions often appears in geometric problems, such as in the related \cite{Han-Jiang}, and each of the terms in the expansion has a precise geometric meaning (some are local, others non-local).

In this general setting, scattering for $g^+$ can be considered  (\cite{Guillarmou}), and one is able to construct the conformal fractional Laplacian on $M$ with respect to the starting $\bar g$ once the log terms in the expansion are controlled. For $s\in(0,1)$ these log terms do not affect the asymptotic expansions at the boundary and one has:

\begin{thm}[\cite{Guillarmou-Guillope} for $s=1/2$, \cite{Chang-Gonzalez} in general]
Fix $s\in(0,1)$. Let $(\overline X^{n+1},\bar g)$ be a smooth compact manifold with boundary $M^n$ and set $h:=\bar g|_{M}$. Let $\rho$ be a geodesic defining function. Then there exists a defining function $\hat \rho$ as in the above construction. Moreover, if $U$ is a solution to the following extension problem
\begin{equation*}\label{div3}
\left\{\begin{split}
-\divergence(\hat \rho^a\nabla U)+E(\hat \rho)U &=0\quad\mbox{in }(\bar X,\bar g), \\
U&=u \quad \mbox{on }M,
\end{split}\right.\end{equation*}
for $E(\hat\rho)$ given in \eqref{E-normal-form}, then the conformal fractional Laplacian $P_s^{h}$ on $M$ with respect to the metric $h$ may be constructed as in Theorem \ref{thm-extension2}.
\end{thm}

One could also look at higher values of $s\in(0,n/2)$. For example, when $M$ is a surface in Euclidean 3-space, one recovers the Willmore invariant with this construction, so an interesting consequence of this approach  is that it produces new (extrinsic) conformal invariants for hypersurfaces in higher dimensions that generalize the Willmore invariant for a two-dimensional surface. Many open questions still remain since this is a growing subject.



\bibliographystyle{abbrv}

\end{document}